\documentclass[10pt]{article}
\usepackage{amssymb,amsfonts}
\usepackage{pb-diagram}
\usepackage{amsmath,amsthm,amsfonts,amssymb}
\usepackage[usenames]{color}
\usepackage{hyperref}
\usepackage{soul}
\usepackage{graphicx}
\usepackage{amssymb}
\usepackage{amsmath}
\usepackage{amssymb}

\newtheorem{theorem}{Theorem}
\newtheorem{definition}{Definition}

\newtheorem{lemma}{Lemma}
\newtheorem{cor}{Corollary}
\newtheorem{remark}{Remark}

\newtheorem{prop}{Proposition}
\newtheorem{problem}{Problem}
\newcommand{\be}{\begin{enumerate}}
\newcommand{\ee}{\end{enumerate}}
\newcommand{\beq}{\begin{equation}}
\newcommand{\eeq}{\end{equation}}

\def\N{{\mathbb{N}}}
\def\Z{{\mathbb{Z}}}

\def\MN{{\mathbb{N}}}

\def\MA{{\mathbb{A}}}
\def\MB{{\mathbb{B}}}
\def\MM{{\mathbb{M}}}
\def\MF{{\mathbb{F}}}

\def\L{{\mathcal{L}}}

{}
{}

\title{Tarski-type problems for free associative  algebras}
\author{Olga Kharlampovich \footnote{Hunter College, CUNY, Supported by  PSC-CUNY  award} and Alexei Myasnikov \footnote{Stevens Institute of Technology, encouraged   by NSF grant DMS-1502254}}

\date{}

\begin{document}

\maketitle

\begin{abstract} In this paper we study fundamental model-theoretic questions for free associative algebras, namely,  first-order  classification, decidability of the first-order theory, and definability of the set of free bases. We show that two free associative algebras of finite rank over  fields are elementarily equivalent if and only if their  ranks are the same and the fields are equivalent in the weak second order logic. In particular, two free associative algebras of finite rank over the same field are elementarily equivalent if and only if they are isomorphic.  We prove that if an arbitrary ring $B$ with  at least one  Noetherian  proper centralizer  is first-order equivalent to a free associative algebra of finite rank over an infinite field then $B$ is also a free associative algebra of finite rank over a field. This solves the elementary classification problem for free associative algebras in a wide class of rings. Finally, we present a formula of the ring language which defines the set of free bases in  a free associative algebra of finite rank.

\end{abstract}

\tableofcontents

\section{Introduction} 
In this paper we give a complete answer to Tarski's-type questions  on  decidability of the first-order theory  and first-order classification for free associative algebras $\MA_K(X)$ with  basis $X$ over a field $K$ in the language of ring theory . Furthermore, we make a major  advance in understanding which arbitrary rings are elementarily equivalent to a given algebra $\MA_K(X)$. We also show that the set of free bases in $\MA_K(X)$ is definable. We obtain these results for unitary, as well as non-unitary, free associative algebras.  This is the first paper in a series on the project on model theory of algebras outlined in our talk  at the ICM in Seoul \cite{ICM}.

Tarski's problems on groups, rings, and other algebraic structures were very inspirational  and led to some  important developments in modern algebra and model theory. 
Usually solutions to  these   problems for some  structure clarify the most fundamental algebraic properties of the structure and give perspective on the expressive power of the first-order logic in the structure.  Indeed, it suffices to mention here results on first-order theories   of algebraically closed fields,  real closed fields \cite{Tarski2}, the fields of $p$-adic  numbers \cite{Ax-Kochen, Ershov}, abelian groups and modules \cite{Szmielewa, Baur}, boolean algebras \cite{Tarski3, Ershov3}, and free and hyperbolic groups \cite{KM1, KM2, Sela, Sela8}. 

In this paper we show that the first-order theory $Th(\MA_K(X))$ of the algebra $\MA_K(X)$ (to avoid trivialities we always assume that $|X| \geq 1$)  is undecidable for any field $K$ and basis $X$ with  (Theorem \ref{th:undecidable}). Furthermore,  algebras  $\MA_K(X)$ and $\MA_L(Y)$ are  first-order (elementarily) equivalent if and only if  either they ranks are finite and equal, or the ranks are both infinite,  and the fields $K$ and $L$ are equivalent in the weak second order logic (Theorem \ref{th:Tarski-unitary}). The latter is a very strong condition on the fields, much stronger then the first-order equivalence. These results in the case of the polynomials   in one variable, i.e., when $|X| = 1$, were known before, see \cite{Robinson,bauval}.  Our main technical tool  is the method of first-order interpretation (see Section \ref{se:interpretation}). We show that the finite rank $|X|$, the arithmetic  $\N = \langle N, +,\cdot, 0\rangle$, and the weak second order theory of the infinite field $K$  are all interpretable in $\MA_K(X)$ uniformly in $K$ and $X$. Here we say that the weak second order theory of  a structure $B$ is  interpretable in $\MA_K(X)$ if the first-order structure $HF(B)$ of hereditary finite sets over $B$, or equivalently, the list superstructure $S(B,\N)$, is interpretable in $\MA_K(X)$ (see Section \ref{se:weak-second-order} for precise definitions). It turns out that  the expressive power of the first-order language of rings is so strong  in $\MA_K(X)$ that it allows one to describe how $\MA_K(X)$ is built from $X$ and $K$. More precisely, on the one hand  the structure $S(K,\N)$ is first-order interpretable in $\MA_K(X)$ (for an infinite $K$), on the other hand,  one can easily construct  an interpretation $\MA^\ast$ of $\MA_K(X)$ in $S(K,\N)$. In fact, one can interpret any "constructive over $K$" algebra $L$ in $S(K,\N)$, but usually this interpretation $L^\ast$ and the original algebra $L$ are not related much. However, in the case of $\MA_K(X)$ we showed that there is a strong relationship between $\MA^\ast$ and $\MA_K(X)$. This relationship  gives a powerful tool to study arbitrary rings which are first-order equivalent to a given algebra $\MA_K(X)$. In particular, we show that if a ring $B$, which has at least one  Noetherian proper centralizer, is first-order equivalent to $\MA_K(X)$ with finite $X$ and infinite $K$, then $B$ is also a free associative algebra over a field $L$ with a finite basis $Y$, in which case $|X| = |Y|$ and $K$ and $L$ are equivalent in the weak second order logic (Theorem \ref{th:elem-classif-unital}). This is an important  breakthrough in our understanding of the first-order properties and  model theory of $\MA_K(X)$. Another result (Theorem \ref{th:bases}) that exploits the established  relationship between  $\MA^\ast$ and $\MA_K(X)$ is that the set of free bases is definable in the algebra $\MA_K(X)$ when the basis $X$ is finite and the field $K$ is infinite.  
In Section \ref{se:non-unitary} we get similar results for non-unital free associative algebras (Theorems \ref{th:undecidable-non-unitary}, \ref{th:Tarski-non-unitary}, \ref{th:elem-classif-non-unital}).  Moreover, we establish a curious connection between para-free associative algebras and first-order classification for $\MA_K(X)$. Namely, we show that every residually nilpotent algebra which is first-order equivalent to $\MA_K(X)$ must be para-free, but which para-free algebras are indeed first-order equivalent to $\MA_K(X)$ remains an open question. We also construct an example of countable algebra which is first-order equivalent to $\MA_K(X)$ but not residually nilpotent. It seems the algebraic structure of such algebras is beyond our  current understanding. However, studying para-free algebras having the same first-order theory as   $\MA_K(X)$ seems like a very interesting project. At the end of the paper we discuss some open problems on this and related subjects.

\section{Preliminaries}

\subsection{Interpretations}\label{se:interpretation}

The model-theoretic technique of interpretation or definability is crucial in our considerations. Because of that we remind here some precise definitions and several known facts that may not be very familiar to algebraists. 

A language $L$ is a triple $(\mathcal{F}_L, \mathcal{P}_L, \mathcal{C}_L)$, where $\mathcal{F}_L = \{f, \ldots \}$ is a  set of functional symbols $f$ coming together with their arities  $n_f \in \mathbb{N}$,  $\mathcal{P}_L$ is  a set of relation (or prediacte) symbols $\mathcal{P}_L = \{P, \ldots \}$ coming together with their arities  $n_P \in \mathbb{N}$,  and a set of constant symbols $ \mathcal{C}_L = \{c, \ldots\}$. Sometimes we write $f(x_1, \ldots,x_n)$ or $P(x_1, \ldots,x_n)$ to show that $n_f = n$ or $n_P = n$.  Usually we denote variables by small letters $x,y,z, a,b, u,v, \ldots$, while the same symbols with bars $\bar x,  \ldots$ denote tuples of the corresponding variables $\bar x = (x_1, \ldots,x_n), \ldots $. A structure in the language $L$ (an $L$-structure) with the base set $A$ is sometimes denoted by $\mathbb{A} = \langle A; L\rangle$ or simply by 
$\mathbb{A} = \langle A; f, \ldots, P, \ldots,c, \ldots \rangle$.  For a given structure $\mathbb{A}$ by $L(\mathbb{A})$ we denote the language of $\mathbb{A}$. Throughout this paper we use frequently the following languages that we fix now: the language of groups $\{ \cdot, ^{-1}, 1\}$, where $\cdot$ is the binary multiplication symbol, $^{-1}$ is the symbol of inversion, and $1$ - the constant symbol for the identity; and the language of rings $\{+, \cdot, 0\}$ with the standard symbols for addition, multiplication, and the additive identity $0$. Sometimes we add the constant $1$ to the language for unitary rings (our rings apriori are not unitary).  When the language $L$ is clear from the context, we follow the standard algebraic practice and denote the structure $\mathbb{A} = \langle A; L\rangle$ simply by $A$. For example, we refer to a field  $\mathbb{F}  = \langle F; +,\cdot,0,1 \rangle$ simply  by $F$, or to a group $\mathbb{G} = \langle G; \cdot,^{-1},1\rangle$ as $G$, etc.

Let $\mathbb{B} = \langle B ; L(\mathbb{B})\rangle$ be a structure. A subset $A \subseteq B^n$ is called {\em definable} in $\mathbb{B}$ if there is a formula $\phi(x_1, \ldots,x_n)$ in $L(\mathbb{B})$ such that  $A = \{(b_1,\ldots,b_n) \in B^n \mid \mathbb{B} \models \phi(b_1, \ldots,b_n)\}$. In this case one says that $\phi$ defines $A$ in $\mathbb{B}$.  Similarly, an operation $f$ or a predicate $P$ on the subset  $A$ is defined in $\mathbb{B}$ if its graph is definable in $\mathbb{B}$. 

In the same vein  an algebraic structure $\mathbb{A} = \langle A ;f, \ldots, P, \ldots, c, \ldots\rangle$  is definable in $\mathbb{B}$ if there is a definable subset $A^* \subseteq  B^n$ and operations $f^*, \ldots, $ predicates $P^*, \ldots, $ and constants $c^*, \ldots, $ on $A^*$ all definable in $\mathbb{B}$ such that the structure $\mathbb{A}^* = \langle A^*; f^*, \ldots, P^*, \ldots,c^*, \ldots, \rangle$ is isomorphic to $\mathbb{A}$.
For example, if $Z$ is the center of a group $G$  then it is definable as a group  in $G$, the same for the center of a ring. 

One can do a bit more in terms of definability. In the notation above if $\sim$ is a definable  equivalence relation on the definable subset $ A \subseteq B^n$ then we say that the quotient set $A^* = A/\sim$ is {\em interpretable} in $\mathbb{B}$.  Furthermore, an operation $f$ or a predicate $P$ on the quotient set $A^*$ is interpretable in $\mathbb{B}$ if the full preimage of its graph in $A$ is definable in $\mathbb{B}$. For example, if $N$ is a normal definable subgroup of a group $G$, then the equivalence relation $x \sim y$ on $G$ given by $xN = yN$ is definable in $G$, so the quotient set $G/N$ of  all right cosets of $N$ is interpretable in $G$. It is easy to see that the multiplication induced 
from $G$ on $G/N$ is also interpretable in $G$. This show that the quotient group $G/N$ is interpretable in $G$. Now we vastly generalize this construction.

\begin{definition} \label{de:interpretable} An algebraic  structure $\mathbb{A} = \langle A ;f, \ldots, P, \ldots, c, \ldots\rangle$  is interpretable  in a structure $\mathbb{B}$  if there is a  subset $A^* \subseteq B^n$  definable in $\mathbb{B}$, an equivalence relation $\sim$ on $A^*$ definable in $\mathbb{B}$, operations  $f^*, \ldots, $ predicates $P^*, \ldots, $ and constants $c^*, \ldots, $ on the quotient set $A^*/\sim$ all interpretable in $\mathbb{B}$ such that the structure $\mathbb{A}^* = \langle A^*/\sim; f^*, \ldots, P^*, \ldots,c^*, \ldots, \rangle$ is isomorphic to $\mathbb{A}$.
 \end{definition}

Interpretation of $\mathbb{A}$ in a class of structures $\mathcal{C}$ is    {\em uniform}  if the formulas that interpret   $\mathbb{A}$ in a structure $\mathbb{B}$ from $\mathcal{C}$   are the same for every structure $\mathbb{B}$ from $\mathcal{C}$.

Sometimes, to define a subset or interpreted a structure $\mathbb{A}$ in a   given structure $\mathbb{B}$ one has to add some elements, say from a subset $P \subseteq B$  to the language $L = L(\mathbb{B})$ as  new constants (we denote the resulting language by $L(\mathbb{B})_P$).    In this case we say that $\mathbb{A}$ is {\em relatively interpretable } or {\em interpretable   with parameters} $P$  in $\mathbb{B}$.  {\it Uniform interpretability with parameters} in a class   $\mathcal{C}$ means that the formulas that interpret   $\mathbb{A}$ in a structure $\mathbb{B}$ from $\mathcal{C}$   are the same for every structure $\mathbb{B}$ from $\mathcal{C}$ and parameters in each such  $\mathbb{B}$ come from  subsets uniformly definable in $\mathcal{C}$.
If we want to emphasize that the interpretability is without constants we say {\em absolutely} interpretable or {\it $0$-interpretable}. In most cases we have the absolute interpretability, so if not said otherwise, throughout the paper interpretability means absolute  interpretability.  
We write $\MA \to_{int} \MB$ when $\MA$ is absolutely interpretable in $\MB$.

The following is a principle result on interpretability. 
\begin{lemma} \cite{Hodges} \label{le:interpr_corol}. If $\MA$ is interpretable in $\MB$  with parameters $P$ then for every formula $\psi(\bar x)$ of $L(\MA)$ one can effectively construct a formula $\psi^*(\bar y, P)$ of $L(\MB)$ such that for any assignment of variables $x_i \to a_i \in \MA$ (so the tuple $\bar x$ goes to a tuple $\bar a$) one has
  $$\MA\models \psi (\bar a)\iff \MB\models \psi^*(\mu (\bar a),P).$$
In particular, for every first-order  sentence $\phi$ in the language of $\MA$  one can effectively construct a sentence $\phi^\ast$ in the language of $\MB$  such that
$$
\MA \models \phi \Longleftrightarrow \MB \models \phi^\ast.
$$
\end{lemma}

The following are two  important corollaries, that we use throughout the paper.

\begin{cor} \label{co:interp} 
\begin{itemize}
\item  If $\MA$ is 0-interpretable in $\MB$ and the first-order theory $Th(\MA)$ is undecidable then  $Th(\MB)$ is also undecidable. 
\item  If $\MA_1$ is 0-interpretable in $\MB_1$ by the same formulas as $\MA_2$ in $\MB_2$ then  $\MB_1 \equiv \MB_2$ implies $\MA_1 \equiv \MA_2$.
\end{itemize}
\end{cor}

Above we discussed properties of the  absolute interpretability, however there is one result on relative interpretability that we  use in the sequel.

\begin{theorem} \cite{Ershov-Lavrov} \label{th:Ershov}
If the  natural numbers $\N = \langle N\mid +,\cdot, 0,1\rangle$ are relatively  interpretable  in $\MB$ then the first-order theory $Th(\MB)$ is undecidable.
\end{theorem}

\begin{definition}
Algebraic structures  $\MA$ and $\MB$ are called {\em bi-interpretable} if  the following conditions hold:
\begin{itemize}
\item $\MB$ is interpretable in $\MA$ as $\MB^*$ (see Definition \ref{de:interpretable} above), $\MA$ is interpretable in $\MB$ as $\MA^*$, which by transitivity  implies that $\MA$ is interpretable in $\MA$, say by $\MA^{**}$, as well as $\MB$ in $\MB$, say as $\MB^{**}$.
\item  There is an  isomorphism $\MA \to \MA^{**}$ which is definable in $\MA$ and there is an isomorphism $\MB \to \MB^{**}$ definable in $\MB$.
\end{itemize}
\end{definition}

\subsection{Weak second order logics}\label{se:weak-second-order} 
\label{se:weak-second-order-intro}

For a set $A$ let  $Pf(A)$ be the set of all finite subsets of $A$.
Now we define by induction  the set $HF(A)$  of hereditary finite sets over $A$;
\begin{itemize}
\item  $HF_0(A)= A$,  
\item $HF_{n+1}(A) = HF_n(A)\cup Pf(HF_n(A))$, 
\item  $HF(A)=\bigcup _{n\in\omega}HF_n(A).$  
 \end{itemize}

For a structure $\MA = \langle A;L \rangle$ define a new two-sorted structure $HF(\MA)$ as follows:
$$
HF(\MA) = \langle \MA, HF(A); \in\rangle,
$$
where the first sort is the structure $\MA$ in the language $L$, the second sort is the set $HF(A)$, and $\in$ is the membership predicate defined on $A \cup HF(A)$. 

One can replace $HF(\MA) $ by a usual first-order structure  as follows. Firstly, one replaces all operations in $L$ by the corresponding predicates (the graphs of the operations) on $A$, so one may assume from the beginning that $L$ consists only of predicate symbols. Secondly, replace  the two-sorted  structure $\langle \MA, HF(A); \in\rangle$ by a structure $\langle A \cup HF(A); L, P_A, \in\rangle$, where $L$ is defined on the subset $A$, $P_A$ defines $A$ in $A \cup HF(A)$, and $\in$ is the membership predicate on $A \cup HF(A)$.  The both structures are "logically equivalent", they both encapsulate the {\em weak second order logic} over $\MA$, i.e., everything that can be expressed in the weak second order logic in $\MA$ can be expressed in the first-order logic in $HF(\MA)$, and vice versa. The structure  $HF(\MA)$ appears naturally in the  weak second order logic, the theory of admissible sets, and $\Sigma$-definability, - we refer to  \cite{B1,B2,E,Ershov2} for details.

 There is another structure, termed the {\em list superstructure} $S(\MA,\MN)$ over $\MA$ whose the first-order theory has the same expressive power as the weak second order logic over $\MA$ and which  is more convenient for us to use in this paper. To introduce $S(\MA,\MN)$ we need a few definitions.  Let $S(A)$ be the set of all finite sequences (tuples) of elements from $A$.
 For a  structure $\MA = \langle A;L \rangle$ define in the notation above  a new two-sorted structure $S(\MA)$ as follows:
$$
S(\MA) = \langle \MA, S(A); \frown, \in\rangle,
$$
where $\frown$ is the binary operation of  concatenation of two sequences from $S(A)$ and $a \in s$ for $a \in A, s \in S(A)$ is interpreted as $a$ being a component of the tuple $s$.  As customary in the formal language theory we will denote the concatenation  $s\frown t$ of two sequences $s$ and $t$ by   $st$.
 
 Now, the structure $S(\MA,\MN)$ is defined as  the three-sorted structure 
 $$
 S(\MA,\MN) = \langle \MA, S(A),\MN; t(s,i,a), l(s), \frown, \in \rangle,
 $$
 where $\N = \langle N\mid +,\cdot, 0,1\rangle$  is the standard arithmetic, $l:S(A) \to N$ is the length function, i.e., $l(s)$ is the length $n$ of a sequence $s= (s_1, \ldots,s_{n})\in S(A)$, and $t(x,y,z)$ is a predicate on $S(A)\times N \times  A$ such that $t(s,i,a)$ holds in $S(\MA,\MN)$ if and only if $s = (s_1, \ldots,s_{n})\in S(A), i \in N, 1\leq i \leq n$, and $a = s_i \in A$. Observe, that in this case the predicate $\in$ is 0-definable in $S(\MA,\N)$ (with the use of $t(s,i,a)$), so sometimes we omit it from the language. Sometimes, in the notation above,  we write $t_i(s) = a$ to indicate that $t(s,i,a)$ holds in $S(\MA,\MN)$ as was described above.  
 
  In the following lemma we summarize some known results (see for example \cite{bauval}) about the structures $HF(\MA), S(\MA)$, and $S(\MA,\N)$.
  
  \begin{lemma} \label{le:superstructures}
  Let $\MA$ be a structure. Then the following holds:
    $$S(\MA) \to_{int} S(\MA,\N) \to_{int} HF(\MA) \to_{int} S(\MA,\N) \to_{int} S(\MA)$$ 
    uniformly in $\MA$ (the last interpretation requires that $\MA$ has at least two elements).
   \end{lemma}
  
  The following result is known, it is based on two facts: the first one is that there are effective enumerations (codings) of the set of all tuples of natural numbers such that the natural operations over the tuples are computable on their codes;  and the second one is that  all computably enumerable predicates over natural numbers  are 0-definable in $\N$ (see, for example, \cite{Coper, Rogers}).  
  
\begin{lemma} \label{le:list-superstructure}
The list superstructure  $S(\N,\N)$ is 0-interpretable in $\N$.
\end{lemma}
We sometimes denote $S(\N,\N)$ by $S^1(\N,\N)$. Notice that the sets of tuples  $S(\N)$ in the interpretation above of $S(\N,\N)$ in $\N$ is a  0-definable subset of $\N$.
 One can consider a set $S(S(\N) \cup \N)$ of all tuples $(s_1, \ldots,s_m)$, where $s_i \in S(\N) \cup \N$ and extend naturally the functions and predicates $t(s,i,a), l(s), \frown, \in$ above to  the set $S(S(\N) \cup \N)$.  This gives a structure 
 $$
 S^2(\N,\N) = \langle S(S(\N) \cup \N),\N; t(s,i,a), l(s), \frown, \in \rangle
 $$
  
  Similar argument to the above gives the following result.
\begin{lemma} \label{le:list-superstructure2}
The structure  $S^2(\N,\N)$ is 0-interpretable in $\N$.
\end{lemma}

One can consider also  structures $S^m(\N,\N)$ for any $m$, but we do not need it in this paper.

 The following result plays an important part in our study of elementary equivalence of free associative algebras. It is  known in folklore, but we put it here with a proof, since we will need the construction in the sequel.

 \begin{theorem} \label{th:int-A(X)-in-S(K,N)}
 Let $X$ be a finite or countable set and $K$ a field. Then the free associative algebra $\MA_K(X)$ with basis $X$ over a field $K$ is 0-interpretable in the structure $S(K,\N)$ uniformly in $K$ and the cardinality of $X$.
 \end{theorem}
\begin{proof}

In  Lemma \ref{le:list-superstructure} we described how one can   0-interpret  the superstructure $S(\N,\N)  = \langle \N, S(\N),\MN; t(s,i,a), l(s), \frown, \in \rangle$  in $\N$. Fix a particular such interpretation and denote it by  $S(\N,\N)^\ast$. This allows us to assume that  the tuples from $S(\N)$ and operations and predicates from $S(\N,\N)$ are 0-interpretable  in $\N$. Furthermore, as was mentioned right after  Lemma \ref{le:list-superstructure} in the  interpretation $S(\N,\N)^\ast$ the set of tuples $S(\N)$ is  interpreted by a 0-definable subset of $\N$ (by the set of the codes of these tuples with respect to some fixed efficient enumeration of the tuples). Since $\N$ is a part of $S(K,\N)$ the argument above gives an interpretation of $S(\N,\N)$ in $S(K,\N)$ uniformly in $K$. Similarly, by Lemma \ref{le:list-superstructure2}, the structure $S^2(\N,\N)$ is interpretable in $S(K,\N)$ uniformly in $K$. 
 
 Let $X = \{x_1, \ldots,x_n\}$ and consider the following interpretation of the free monoid $\MM_X$  in $S(\N,\N)$.  A monomial $M=x_{i_1}\ldots x_{i_m}\in\MM_X$ can be uniquely  represented by a tuple of natural numbers $ t_M=(i_1,\ldots, i_m)$.  Here we assume that the identity 1 in $\MM_X$ (the trivial monomial) is represented by the number 0. Denote by $T$ the set of all tuples $t= (t_1, \ldots,t_m) \in S(\N)$, $m\in \N$,  such that for any $i$ one has $1\leq t_i \leq n$. Here again we assume that tuples of length 0 are all equal to each other and represented by the number 0.   Conversely,  with  any tuple $t= (t_1, \ldots,t_m) \in T$  one can associate a  monomial $M_t = x_{t_1} \ldots x_{t_m} \in \MM_X$ (here $M_0 = 1$ in $\MM_X$). The multiplication in $\MM_X$  corresponds to  concatenation of  tuples  in  $T$, which is 0-definable in $S(\N,\N)$. The construction above gives a  0-interpretation of $\MM_X$  in $S(\N,\N)$. Combining this interpretation with the interpretation  $S(\N,\N)^\ast$ of  $S(\N,\N)$ in $\N$ one gets a 0-interpretation of $\MM_X$ in $\N$, hence in $S(K,\N)$, which we denote by $\MM_X^*$. Observe that the map $M \to t_M$ gives rise to an isomorphism $\MM_X \to \MM_X^*$,   while the map $t \to M_t$ - to the inverse isomorphism $\MM_X^* \to \MM_X$.  Notice that the 0-interpretation $\MM_X^*$ of $\MM_X$  in $S(K,\N)$ is uniform in $K$ and $n = |X|$. 
 
 Building on the interpretation  $\MM_X^*$ of $\MM_X$ in $S(\N,\N)$  we interpret $\MA_K(X)$ in $S(K, \N)$ as follows.  For an element  $f=\sum_{i= 1}^e \alpha _iM_i \in\MA_K(X)$, where $\alpha_i \in K, M_i \in \MM_X$,  we associate a pair  $q_f = (\overline\alpha, \overline t)$, where $\overline\alpha = (\alpha_1, \ldots, \alpha_e)$, $\overline t = (t_{M_1}, \ldots,t_{M_e})$.  
 Recall, that the representation $f=\sum_{i= 1}^e \alpha _iM_i \in\MA_K(X)$ above is the {\em reduced form} of $f$ if $\alpha_i \neq 0$ (unless $e = 1$ and $M_1 = 1$ in $\MM_X$) and $M_i \neq M_j$ for $i \neq j$. A reduced  form of $f$ is unique up to a permutation of summands.

 Denote by  $S(T)$ the set of all tuples of elements of $T$ and by $S(T)^0$ the subset of all tuples $s = (t_1, \ldots,t_e)$ such that $t_i \neq t_j$. By $C$ we denote the set of tuples $a = (\alpha_1, \ldots, \alpha_e)$ of elements from $K$, and by $C^0$ the subset of all such tuples  where $\alpha_i \neq 0$.  Finally, put
 $$
A = \{ (a,s)    \mid  a \in C, s \in S(T), \ell(a) = \ell(s) \}
$$
$$
 A^0 = \{ (a,s)    \mid  a \in C^0, s \in S(T)^0, \ell(a) = \ell(s) \} 
  $$
 here $\ell(a)$ and $\ell(s)$ are the lengths of the tuples $a, s$. One can view the set $A$ as a set which "represents"  non-zero elements of $\MA_K(X)$.  Namely, a pair  $(a,s) =((\alpha_1, \ldots,\alpha_m),(t_1, \ldots,t_m))$ represents a polynomial $f_{(a,s)} = \sum_i\alpha_iM_{t_i}$. In this vein,  the set $A^0$ represents the set of "reduced  forms" of non-zero elements of   $\MA_K(X)$.   The map $\tilde A \to \MA_K(X)$  which maps $(a,s) \to f_{(a,s)}$ is onto the set of all non-zero elements from  $\MA_K(X)$. 
 
  After we formally add "zero" to $A$  by setting   $\tilde A = A \cup \{((0),(0))\}$,  where $ ((0),(0))$ is the tuple that corresponds to the element $0\cdot1 = 0$ in $\MA_K(X)$ (as before $(0)$ is the tuple that corresponds to 1 in $\MA_X$, since $0$ corresponds to 1 in $\MM_X$), the map $\tilde A \to \MA_K(X)$,  which maps $(a,s) \to f_{(a,s)}$ and $((0),(0)) \to 0$ is onto. 
  
 To interpret  $\MA_K(X)$ in $S(K,\N)$ we need to define by formulas of $S(K,\N)$   an equivalence relation $\sim$ on $A$ such that $(a,s) \sim (a_1,s_1)$ if and only if $f_{(a,s)} = f_{(a_1,s_1)}$ in $\MA_K(X)$.  We do it in several  steps.
 
Firstly, we define by formulas the restriction $\sim_0$ of $\sim$ onto $A^0$.
 Notice that   $(a,s) \sim_0 (a_1,s_1)$ on $A^0$  if and only if the following conditions hold:
 \begin{itemize}
 \item  the length of $s$ is equal to the length of $s_1$, 
 \item  every component of $s$ is equal to some component of $s_1$,  and vice versa,
 \item  if $i$'s component of $s$ is equal to $j$'s component of $s_1$ then $i$'s component of $a$ is equal to $j$'s component of $a_1$,
 \end{itemize}
 and these conditions can be written by a  formula, say $\psi(x_1,x_2,y_1,y_2)$  in the language of $S(K,\N)$ (using the predicate $t(s,i,a)$ and the function $\ell(s)$ from the definition of $S(K,\N)$).  We extend $\sim_0$ onto $ A \cup \{((0),(0))\}$
by  setting that $((0),(0))$ is equivalent only to itself. 
 
 Now we show how to describe by formulas  the reduced forms of an element $(a,s) \in A$. To do this we use the structure $S^2(\N,\N)$ which is 0-iterpretable in $S(K,\N)$.   Given $(a,s) \in A$,  we define several tuples as follows. Let $a = (\alpha_1, \ldots,\alpha_m), s = (s_1, \ldots,s_m)$. 
 Define $s^* = (s_1^*, \ldots,s_m^*) \in S(S(T) \cup T)$ such that
 \begin{itemize}
 \item $s_1^* = s_1$,
 \item for any $i \leq m$ if there is $j \leq m$ such that $s_j = s_{i+1}$ then $s_{i+1}^* = s_i^*$,
 \item for any $i \leq m$ if there is  no $j \leq m$ such that $s_j = s_{i+1}$ then $s_{i+1}^* = s_i^* \frown (s_{i+1}) $, where $(s_{i+1})$ is a tuple of length 1 with the component $s_{i+1}$.
 \end{itemize}
 Clearly, the last component $s_m^*$  of $s^*$   is a tuple which  formed by components of $s$ which are taken in the same order as in $s$ but without repetition. Let $s_m^* = s' = (s_1',\ldots,s_{m'}')$. Notice that $s' \in S(T)^0$ and $s'$ is definable in $S(K,\N)$.
 
 Now for any $p, 1 \leq p \leq \ell(s^*)$, we define a tuple $b^{(p)} = (\beta_{1}^{(p)}, \ldots,\beta_{m}^{(p)})$ such that 
 \begin{itemize}
 \item $\beta_1^{(p)} = \alpha_1$ if $s_1 = s_p'$, otherwise $\beta_1^{(p)} = 0$,
 \item for any $i, 1\leq i \leq m$  if $s_{i+1} = s_p'$ then $\beta_{i+1}^{(p)} = \beta_i^{(p)} +\alpha_{i+1}$,
 \item for any $i, 1\leq i \leq m$  if $s_{i+1} \neq  s_p'$ then $\beta_{i+1}^{(p)} = \beta_i$
 \end{itemize}
 The element $\beta_m^{(p)} \in K$ is the coefficient of  the monomial $M_{s_p'}$ of the element $f_{(a,s)}$ when one collects the similar terms. Clearly, the conditions above can be written by formulas in the language of $S(K,\N)$.
 Put $b' = (\beta_m^{(1)}, \ldots, \beta_m^{(m')} )$. The tuple $b'$ is definable in $S(K,\N)$ with parameters $a, s$, i.e., there is a formula $\psi_2(x_1,x_2,y_1)$ in the language of $S(K,\N)$ such that  $\psi_2(a,s,b)$ holds in $S(K,\N)$ on some $b$ if and only if $b = b'$. Furthermore, note that $f_{(a,s)} = f_{(b',s')}$.   
 
 The pair $(b',s')$ is not reduced since some components of $b'$ (the coefficients of the monomials in $f_{(b',s')}$ might be zero. To remove these zeros we define two tuples $\hat a = ({\hat a}_1, \ldots, {\hat a}_{m'})$ and ${\hat s} = ({\hat s}_1, \ldots,{\hat s}_{m'})$  such that:
 
 \begin{itemize}
  \item if $b' = (\beta_m^{(1)}, \ldots, \beta_m^{(m')} )$ is such that $\beta_m^{(j)} = 0$ for all $j, 1\leq j \leq m'$ then ${\hat a} = (0)$ and ${\hat s} =(0)$, so the tuple $(\hat a, \hat s)$ represents zero in $\MA_K(X)$. Otherwise, ${\hat a}_1$ is a tuple of length one  equal to $(\beta_m^{(j)})$ where $j$ is the least one such that $1 \leq j \leq m'$ and $\beta_m^{(j)} \neq 0$.
  \item for any $i, 1 \leq i \leq m'$ if $\beta_m^{(i+1)} \neq  0$ then ${\hat a}_{i+1} = {\hat a}_i \frown  (\beta_m^{(i+1)})$ and ${\hat s}_{i+1} = {\hat s}_i \frown  (s_{(i+1)}')$.
  \item   for any $i, 1 \leq i \leq m'$ if $\beta_m^{(i+1)} = 0$ then ${\hat a}_{i+1} = {\hat a}_i $ and ${\hat s}_{i+1} = {\hat s}_i$.
 \end{itemize}
 The pair $({\hat a}, {\hat s})$  is definable in $S(K,\N)$ with parameters $a, s$. Notice also that the pair $({\hat a}_{m'},{\hat s}_{m'})$ formed by the last components of the tuples $\hat a$ and $\hat s$ gives a reduced form of the polynomial $f_{(a,s)}$.  We denote the tuples ${\hat a}_{m'}$ and ${\hat s}_{m'}$ by  $red(a)$ and $red(s)$, respectively. Now  $f_{(a,s)} = f_{({red(a)}, {red(s)})}$. The reduced pair of the trivial  pair $((0),(0))$ is the pair $((0),(0))$ itself.
 
 Now we can define by formulas the equivalence relation $\sim$ on $A$. Namely, $(a,s) \sim(a_1,s_1)$ if and only if there exists pairs $(red(a), red(s))$ and $(red(a_1), red(s_1) )$ that satisfy, correspondingly, the conditions above and such that  $(red(a), red(s)) \sim_0 (red(a_1), red(s_1) )$. To finish the definition of $\sim$ on $A$ it suffices to add that the pair $((0),(0))$ is equivalent only to itself.
 
 Now one needs to define by formulas an addition $\oplus$ and multiplication  $\odot$ on $\tilde A/\sim$ that would correspond to the ring operations on $\MA_K(X)$.  In fact, it suffices to define  $\oplus$ and $\odot$ on $A/\sim$ and  then extend it to  $\tilde A/\sim$ in the obvious way. Let $(a_i,s_i) \in A$, $i = 1,2,3$, where $a_i =(a_1^{(i)}, \ldots,a_{\ell(a_i)}^{(i)}), s_i =(s_1^{(i)}, \ldots,s_{\ell(s_i)}^{(i)})$. Define $\oplus$ by
 $$
 (a_1,s_1) \oplus (a_2,s_2) = (a_1 \frown a_2, s_1 \frown s_2).
 $$
 To define $\odot$ we need to fix a computable function $\pi(i,j,x,y)$ such that for any non-zero $p,q \in \N$ the function $\pi(p,q,x,y)$ gives a bijection 
 $$
\pi(p,q,x,y) : \{(x,y) \mid 1 \leq x \leq p, 1 \leq y \leq q\}  \to \{1, \ldots,pq\}
 $$
 Since $\pi$ is computable there is a formula that defines the graph of $\pi$   in $\N$. Now put
 $$
  (a_1,s_1) \odot (a_2,s_2) = (a_3,s_3),
 $$
 where $a_3, s_3$ satisfy the following conditions: 
 
 \begin{itemize}
 \item $\ell(a_3) = \ell(s_3) = \ell(a_1)\ell(a_2)$,
 \item for any $r, 1 \leq r \leq \ell(a_3)$ if $\pi(\ell(a_1),\ell(a_2),x,y)  =  r$ then  $a_r^{(3)} = a_x^{(1)} a_y^{(2)}$ and $s_r^{(3)} = s_x^{(1)} \frown  s_y^{(2)}$
 \end{itemize}

 This gives a 0-interpretation of $\MA_K(X)$ in $S(K,\N)$ uniformly in $K$ and the cardinality of $X$ in the case when the set $X$ is  finite. By a slightly modified argument  one can interpret $\MA_K(X)$ in $S(K,\N)$ uniformly in $K$ and the cardinality of $X$ in the case when the set $X$ is countable. This proves the theorem.

\end{proof}

\subsection{Fields equivalent in the weak second order logic}
\label{se:HF-fields}

The weak order logic is quite  powerful. Indeed,  unlike  the first-order logic, there are many infinite algebraic structures which are completely characterized by their weak second order logic.  In particular, each of the following fields is determined up to isomorphism by its  weak second-order theory: ${\mathbb Q}$, finitely generated algebraic  extensions of  ${\mathbb Q}$,   algebraically closed fields of finite transcendence degree over their prime subfields, pure transcendental finite  extensions of a prime field,  the field of algebraic real numbers.  These results are known in the folklore, the proofs are based on the Gandy's theorem on fixed points of $\Sigma$-definable operators (see, for example, \cite{E, B2}),  some of these results are also mentioned  in \cite{bauval} (Corollary V.2.8). However, we could not find any precise references in the literature. 

Fields equivalent in the weak second order logic play a crucial part in the first-order classification of free associative algebras.

\section{Maximal rings of scalars and algebras}

Let $R$  be a commutative associative ring with unity 1, and $M, N$ exact $R$-modules.  Let $f:M \times M \to N$ be an $R$-bilinear map.  For a subset $E \subseteq M$ we define the left and right annulators of $E$ by $Ann_l(E) =\{x \in M \mid f(x,E) = 0\} $ and $Ann_r(E) = \{y \in M \mid f(E,y)=0\} $.

 We say that 
\begin{itemize}
\item [1)] $f$ is  {\em non-degenerate} if $Ann_l(M) =Ann_r(M) = 0$. 
\item [2)] $f$ is {\em onto}   if the submodule (equivalently,  the subgroup) $\langle f(M,M)\rangle$ generated by $f(M,M)$ is equal to $N$.
\item [3)] $f$ has a {\em finite complete system} if there is  a finite subset $E \subseteq M$ (called a {\em  complete system} for $f$) such that $Ann_l(E) = Ann_l(M)$ and $Ann_r(E) = Ann_r(M)$.
\item [4)] $f$ has {\em finite width }  if   there exists some natural number $m$ such that for any $z \in N$ there are some $x_i, y_i \in M, i = 1, \ldots,m$ such that $z=\sum _{i=1}^mf(x_i,y_i)$. The least such  $m$ is termed the width of $f$.
\end{itemize}

Note that the conditions 1) - 4) do not depend on the ring $R$, i.e., whether they hold or not in $f$ depend only on the abelian group structure of $M$ and $N$. 

Let $L$ be an $R$-algebra (not necessary associative) over the ring $R$. Denote by $L^2$ the $R$-submodule of $L$ generated by all products $xy$ where $x, y \in L$. Then the multiplication map   $f_L:L\times L\rightarrow L^2$ is $R$-bilinear and onto.  This map induces a non-degenerate $R$-bilinear onto map ${\bar f}_L: L/Ann_l(L)  \times L/Ann_r(L) \to L^2$, where $Ann_l(L)  = \{x \in L \mid xL = 0\}, Ann_r  = \{y \in L \mid Ly = 0\}$.

\begin{lemma} \label{le:1-4} Let $L$ be a finitely generated  $R$-algebra which is either associative or Lie. Then the bilinear map ${\bar f}_L$ satisfies all the conditions 1)-4). In particular, if $Ann_l(L) = Ann_r(L) = 0$ then the multiplication $f_L$ satisfies all the conditions 1)-4).
\end{lemma}
\begin{proof}
Suppose $L$ is generated (as an algebra) by a finite set $X$.  The map ${\bar f}_L$ satisfies conditions 1) and 2) by construction. To prove 3) it suffices to show that $Ann_l(L) = Ann_l(X)$ and $Ann_r(L) = Ann_r(X)$. We prove the first equality (the second one is similar). Let $a \in Ann_l(X)$ and $b \in L$. To show that $ab = 0$ we may assume by linearity that $b$ is a product of elements from $X$. If $b \in X$ then $ab = 0 $, otherwise, $b = uv$, where $u, v$ are products of elements of $X$ of shorter length. By induction on length $au = av = 0$. If $L$ is associative then $a(uv) = (au)v = 0$. If $L$ is Lie then $a(uv) = -u(va) -v(au) = u(av)-v(au) = 0$, hence the claim. To show 4)  we  prove that $L = Lx_1 + \ldots + Lx_n$, where $X = \{x_1, \ldots, x_n\}$. Clearly, it suffice to show that every product $p$ of elements from $X$ belongs to $M =  Lx_1 + \ldots + Lx_n$. If $L$ is associative then every such product $p$ ends on an element from $X$, so the claim holds. If $L$ is Lie then $p = uv$ for some Lie words $u, v$ in $X$. We use induction on the length of $v$ (as a Lie word in $X$) to show that $p \in M$. If $ v$ is an element from $X$ then there is nothing to prove.  Otherwise, $v = v_1v_2$ where $v_1, v_2$ are Lie words in $X$ of smaller length. Then $u(v_1v_2) = - v_1(v_2u) -v_2(uv_1) = (v_2u)v_1 + (uv_1)v_2$. Now by induction on the length of the second factors we get that 
$(v_2u)v_1, (uv_1)v_2$, and hence $(v_2u)v_1 + (uv_1)v_2$, are in $M$, as required.
\end{proof}

For any non-degenerate onto bilinear map $f:M \times M \to N$ there is a uniquely defined {\em maximal ring of scalars} $P(f)$, which is an analog of the {\em centroid} of a ring.  More precisely, a commutative associative unitary ring $P$ is called a "ring of scalars" of $f$ if $M$ and $N$ admit the  structure  of exact  $P$-modules such that $f$ is $P$-bilinear. A ring of scalars $P$ of $f$ is called {\em maximal} if for every ring of scalars $P^\prime$ of $f$ there is a monomorphism $\mu:P^\prime \to P$ such that for every $\alpha \in P^\prime$ its actions on $M$ and $N$ are the same as the actions of $\mu(\alpha)$. It was shown in \cite{M} that the maximal ring of scalars of $f$  is unique up to isomorphism, as well as its actions on $M$ and $N$. We denote it by $P(f)$. In fact, the ring $P(f)$ can be constructed as follows. 

Let $End(M)$ be the ring of endomorphisms of $M$ (here $M$ is viewed as an abelian group). Denote by $Sym_f(M)$ the subgroup of all $f$-symmetric endomorphisms $A \in End(M)$, i.e. such that $f(Ax,y) = f(x,Ay)$ for any $x, y \in M$. Let $Z$ be the center of $Sym_f(M)$, which is the subgroup of $Sym_f(M)$  consisting of all endomorphisms $A$ that commute with every endomorphism in $Sym_f(M)$.  For every natural number $n$ denote by $Z_n$ the subset of those elements $A \in Z$ such that for any $x_i,y_i, u_i, v_i \in M, i = 1, \ldots,n$ the following condition holds:
\begin{equation} \label{eq:Z}
\sum_{i = 1}^n f(x_i,y_i) = \sum_{i = 1}^n f(u_i,v_i) \longrightarrow \sum_{i = 1}^n f(Ax_i,y_i) = \sum_{i = 1}^n f(Au_i,v_i).
\end{equation}
Finally, define 
$$P(f) = \cap_{n = 1}^\infty Z_n.$$
Straightforward verification shows that $Z_n$, as well as $P(f)$,  is a commutative associative unitary subring of $End(M)$, so $M$ is an exact $P(f)$-module. The conditions (\ref{eq:Z}) allows one to define the action of $P(f)$ on the submodule of $N$ generated by $f(M,M)$, which is the whole module $N$, since $f$ is onto. It is not hard to see that $P(f)$ is a maximal ring of scalars of $f$.

To study model theory of $f:M \times M \to N$ one associates with $f$ a two-sorted structure ${\mathcal A}(f) = \langle M, N; f \rangle$, where $M$ and $N$ are abelian groups equipped with the map  $f$ (the language of ${\mathcal A} (f)$ consists of additive group languages for $M$ and $N$, and the predicate symbol for the graph of $f$).

\begin{theorem} \label{th:bilinear} \cite{M} Let $f$ be a $K$-bilinear map $M\times M\rightarrow N$ that satisfies 1)-4)  above. Then the maximal ring of scalars  $P(f)$ for $f$ and its actions on $M$ and $N$ are 0-interpretable  in ${\mathcal A}(f)$ uniformly in the size of the finite complete system and the width of $f$.
\end{theorem}

\begin{prop}\label{th:scalar} Let $L$ be a non-commutative free associative (unital or not) or a  non-commutative free Lie algebra over a field $K$,  or a group ring of a non-commutative torsion-free hyperbolic group over a field $K$.  Then the maximal ring of scalars $P(f_L)$ of the multiplication bilinear map $f_L$ is isomorphic  to the field  $K$.
\end{prop}

\begin{proof}  

There are three cases to consider for the algebra $L$: associative, Lie, and the group ring. Notice that in all of them $Ann_l(L) = Ann_r(L) = 0$ and $f_L$ is onto (see Lemma \ref{le:1-4}), so the maximal ring of scalars $P = P(f_L)$ exists.

Case 1. Let $L = A_K(X)$ be a free  associative algebra over a field $F$. Let $\alpha\in P$.
The action of $\alpha$ on $L$ gives rise to  a  $K$-endomorphism, say $\phi _{\alpha}$ of $L$, viewed as a $K$-module. We have $(\alpha x)y=x(\alpha y)$ for any $x,y\in A$.  Assume now that $x$ and $y$ are distinct letters from $X$. Therefore $\phi _{\alpha}(x)y=x\phi _{\alpha}(y).$ This implies 
$\phi _{\alpha}(x)=xu, \ \phi _{\alpha}(y)=vy$ for some $u, v \in L$. But then $xuy = xvy$ so  $u=v$.  
One has  $\phi_\alpha(xx) = \phi _{\alpha}(x)x=x\phi _{\alpha}(x)$, therefore $xux=xxu$ and $xu=ux.$ Similarly $uy=yu$.  By Bergman's theorem the centralizers $C_L(x)$ and $C_L(y)$ are equal, correspondingly, to  the rings of polynomials $K[x]$ and $K[y]$, hence $u \in K[x] \cap K[y] =  K.$ It follows that $\phi_\alpha$ acts on $x$ and $y$ as some scalar $u$ from $K$. Replacing $y$ by an arbitrary letter $z \in X$ in the argument above one gets that the action of $\phi_\alpha$ on every element from $X$ is by the scalar $u \in K$. Since every product $p$ of elements  from $X$ is either a letter from $X$ or a product of the type $p = xp^\prime$, where $x \in X$, one has $\phi_\alpha(p) = \phi_\alpha(x)p^\prime = (ux)p^\prime = u(xp^\prime) = up$. By linearity $\phi_\alpha$ acts on $L$ by multiplication by the scalar $u$, so $\phi_\alpha = \phi_u$. Since $L$ is an exact $P(f_L)$-module this implies  that $P(f_L) = K$, as claimed. 

Case 2. Let $L$ be a free Lie algeba over $K$ with basis $X$. It is known (see, for example \cite{Magnus}) that for any $x \in X$ and $a  \in L$ if $[a,x] =0$ then $a \in Kx$. Let $\alpha \in P$ then the action of $\alpha$ on $L$ gives a $K$-endomorphism $\phi_\alpha$ of $K$-module $L$ such that $\phi_\alpha(xy) = \phi_\alpha(x)y = x\phi_\alpha(y)$. In particular, $\phi_\alpha(xx) = 0 = (\phi_\alpha(x)x)$, so $\phi_\alpha(x) \in Kx$, say $\phi_\alpha(x) = \alpha_xx$, where $\alpha_x \in K$. Similarly, for $y \in X$ $\phi_\alpha(y) = \alpha_yy$ for some $\alpha_y \in K$. It follows that $\phi_\alpha(xy) = \alpha_x (xy) = \alpha_y (xy)$, hence $\alpha_x = \alpha_y$ for any $x, y \in X$. Therefore, $\phi_\alpha$ acts on $L$ precisely by multiplication of $\alpha_x$. This shows that $P = K$.

Case 3. Let  $L = K(G)$ be a group algebra of a torsion-free hyperbolic group $G$ over a field $K$. Suppose $P$ is a maximal ring of scalars of $L$ and $\alpha \in P$. Then as before $\alpha$ gives rise to  a $K$-linear endomorphism $\phi_\alpha$ of $L$ viewed as a $K$-module. It follows  that for a given non-trivial  element $g \in G$ one has $\phi _{\alpha}(g)g=g\phi _{\alpha}(g)$, so $\phi _{\alpha}(g)\in C_{KG}(g)$. If $g$ is not a proper power then $C_{KG}(g) =K[g, g^{-1}]$ - the ring of Laurent polynomials in one variable $g$.  Therefore $\phi _{\alpha}(g)=\sum_{i \in I} \gamma _ig^{i}$ for some finite subset $I \subset \mathbb{Z}$ and $0 \neq \gamma_i \in K$ for  $i\in I$. Similarly for a non-trivial $h \in G$, which is not a proper power in $G$, and  such that $[g,h] \neq 1$ one has 
$\phi _{\alpha}(h)=\sum_{j \in J} \sigma _jh^{j}$ for some finite subset $J \subset \mathbb{Z}$ and $0\neq \sigma_j \in K$ for $j \in J$. Note that $\phi_\alpha(gh) = \phi_\alpha(g)h = g\phi_\alpha(h)$, so  $\sum_{i \in I} \gamma _ig^{i}h=\sum_{j \in J} \sigma _jgh^{j}$. This implies that there is a bijection $\theta: I \to J$  such that $g^ih = gh^{\theta(i)}$ and $\gamma_i = \sigma_{\theta(i)}$ for each $i \in I$ (indeed, since $G$ is torsion-free $g^ih \neq g^kh$ for any $i \neq k$, as well as $gh^j \neq gh^k $ for $j \neq k$). Hence $g^{i-1} = h^{\theta(i)-1}$  for every $i$. Recall that the centralizers of non-trivial elements in a torsion-free hyperbolic group $G$ are infinite cyclic, so the commutativity relation on non-trivial elements from $G$ is transitive. Since $g$ and $h$ do not commute and are of infinite order  the  equality above may happen only if $i = 1$ and $\theta(i) = 1$.  Hence $I = \{1\} = J$, so 
 $\phi _{\alpha}(g)=\gamma_1g$,  hence  $\phi _{\alpha}(h)=\sigma_1h,$ with $\gamma_1 = \sigma_1$, which we now denote by $\gamma$. Since $h$ was an arbitrary  non-trivial not a proper power element in $G$ it follows $\phi_\alpha(h) = \gamma h$. It is known that every non-trivial element in a torsion-free hyperbolic group has a unique maximal root, so every for every $1 \neq g \in G$ there is a unique  positive integer  $n_g$ and a unique element $g_0 \in G$, which is not a proper power,  such that $g = g_0^{n_g}$. This shows that $\phi_\alpha(g) = \phi_\alpha(g_0) g_0^{n-1} = \gamma g_0^n =\gamma g$. Hence $\phi_\alpha$ acts on $G$ precisely by multiplication by the scalar $\gamma \in K$. By linearity it acts on the whole algebra $L$ by multiplication by $\gamma$, so $\phi_\alpha = \phi_\gamma$, as required.
 
\end{proof}

\begin{remark} \label{re:max-ring}
If $L$ is a commutative free associative (unital or not) algebra over a field $K$ then the maximal ring of scalars of $L$ is isomorphic to the ring of commutative polynomials $K[x]$ in one variable $x$. 
\end{remark}
Indeed, if $L$ is unitary, i.e., $L = \MA_K(X)$ where $X = \{x\}$ is a singleton, then $P =  K[x]$ (the action of $\alpha \in P$ is completely defined by its action on $1$). If $L$ is non-unital then  $L = \MA_K^0(x)$ (see Section \ref{se:non-unitary}) and for every $\alpha \in P(L)$ its  action on $L$ is completely defined by the image $\phi_\alpha(x)$ since every element in $L$ is divisible by $x$. Note that in this case  $\phi_\alpha(x) = xu$ for some $u \in K[x]$. This gives $P = K[x]$.

 From Theorem \ref{th:bilinear} and Proposition \ref{th:scalar} we get the following result.
 
\begin{theorem} \label{th:field} Let $L$ be  non-commutative  either a free associative (unital or not) of finite rank over a field $K$, or a free Lie algebra  of finite rank over a field $K$,  or a group algebra of a torsion-free hyperbolic group over a field $K$.  Then the field $K$ and its action on $L$ is 0-interpretable in $L$. 
\end{theorem}
\begin{proof}
It is easy to see that $Ann_l(L) = Ann_r(L) = 0$. By Lemma \ref{le:1-4} the bilinear map $f_L$ satisfies all the conditions 1)-4). Notice that the additive  groups $L$ and $L^2$ are definable in $L$ (definability of $L^2$ follows from the property 4)), as well as the ring multiplication $f_L:L\times L \to L^2$. Hence the structure $\mathcal{A}(f_L)$ is interpretable in $L$. By Theorem \ref{th:bilinear} the maximal ring of scalars $P(f_L)$ and its action on $L$ is interpretable in $\mathcal{A}(f_L)$, hence in $L$. Now by Proposition \ref{th:scalar} the ring $P(f_L)$ is isomorphic to $K$, and the result follows.
\end{proof}

\begin{remark} \label{re:def-field}
If $L$ is a commutative  free associative (unital or not) algebra over a field $K$ then the filed $K$ and its action on $L$ is definable in $L$.
\end{remark}
Indeed, it follows from Remark \ref{re:max-ring} that $P(L)  \simeq K[x]$, hence the polynomial ring $K[x]$ is 0-interpretable in $L$. It remains to note that $K$ is 0-definable in $K[x]$.

\section{Definability in polynomial rings}

 For the rest of this section we fix the following notation. Let $F$ be a field, $X$ a set of variables,  and $F[X]$ a ring of commutative polynomials 
 with variables in  $X$ and  coefficients in $F$.  In this section we discuss interpretability of various objects in the ring $F[X]$. Many of the results of this section are known, especially on interpretability with parameters, but for our purposes   we usually need them in a much stronger form - when the isomorphisms between such interpretations with parameters are first-order definable uniformly in the parameters.  Besides, we prove that various different interpretations of the same structure, say the arithmetic $\N = \langle N; +,\cdot,0,1\rangle$, in $F[X]$ have canonical isomorphisms uniformly definable in $F[X]$.  In what follows, if not said otherwise,  the terms  {\em definable} and {\em interpretable} mean $0$-definable  and  0-interpretable.

\subsection{Basic facts}

We start with   the following obvious results. Recall that a polynomial $a \in F[X]$ is called {\em irreducible} if it is not invertible (non-constant) and if $a = uv$ for some $u, v \in F[X]$ then either $u \in F$ or $v \in F$.

\begin{lemma} \label{le:field}
Let $F$ be an arbitrary field and $X$ an arbitrary non-empty set.  Then the following hold:
\begin{itemize}
\item [1)] The field $F$ is $0$-definable  in $F[X]$ uniformly in $F$.
\item [2)] The set $Irr$ of all irreducible polynomials is $0$-definable in  $F[X]$.
\end{itemize}
\end{lemma}
\begin{proof}
The field $F$, as a subset  of $F[X]$ consists precisely of all invertible elements of $F[X]$, so it  can be described by a first-order formula 
$$\phi(x) = \exists y (xy = 1)$$
 that does not depend on $F$. This proves 1).

 The set $Irr$ of all irreducible polynomials in $F[X]$ is definable in $F[X]$ by the formula
$$
Irr(x) = \forall u \forall v  (x = uv \rightarrow u \in F  \vee v\in F) \wedge (x \not \in F).
$$
\end{proof}
In view of Lemma \ref{le:field} we will use notation $a \in F$ meaning that $a \in F[X]$ satisfies the corresponding  formula from Lemma \ref{le:field}.  In a commutative ring $R$ for elements  $x,y \in R$ we write  $x\mid y$ if $y = xz$ for some $z \in R$. Obviously, this is also a definable predicate in the language of rings, so we can use it in our formulas.  We frequently use the fact that $F[X]$ is a unique factorization domain without mentioning it directly.

\begin{lemma} \label{le:int-poly-one}
Let $F$ be an arbitrary field and $X$ an arbitrary non-empty set. Let $P$ be a non-invertible polynomial in $F[X]$. Then the ring of polynomials in one variable $F[P]$ is definable  in $F[X]$ uniformly in $F$, $X$,  and  $P$.
\end{lemma}
\begin{proof}
 Fix a non-invertible polynomial $P \in F[X]$. 
 The following formula with the parameter $P$ defines the ring of polynomials $F[P]$ in $F[X]$:
 $$
 \psi(Q,P) = \forall \alpha \in F \exists \beta \in F (P-\alpha \mid Q-\beta).
 $$
 Indeed, any $Q \in F$ satisfies the formula for $\beta = Q$. Suppose now  $Q\in F[P] \smallsetminus F$ then for any $\alpha \in F$ $Q = (P-\alpha)Q_1 +\beta$ for some $\beta \in F$. Hence, $P-\alpha \mid Q-\beta$, so $Q$ satisfies $\psi(Q,P)$ in $F[X]$. 
 
 On the other hand, if $F[X]\models \psi(Q,P)$ for some $Q \in F[X]$, then for a given $\alpha \in F$ one has $Q -\beta = (P-\alpha)Q_0$ for some $\beta \in F$ and $Q_0 \in F[X]$. For another $\alpha_1\in F$ there exists $\beta_1 \in F$ such that  
 $(P-\alpha_1)\mid Q-\beta_1$. Now, 
 $$
 Q-\beta = (P-\alpha)Q_0 =  (P-\alpha_1 +\alpha_1- \alpha) Q_0= 
 (P-\alpha_1)Q_0 +(\alpha_1-\alpha)Q_0.
 $$
  Hence 
  $$Q-\beta_1 = Q-\beta +\beta -\beta_1 =  (P-\alpha_1)Q_0 +(\alpha_1-\alpha)Q_0 + \beta -\beta_1. 
   $$
    It follows that $P-\alpha_1 \mid (\alpha_1-\alpha)Q_0 + \beta -\beta_1$, and  $P-\alpha_1 \mid Q_0 + (\beta -\beta_1) (\alpha_1-\alpha)^{-1}$, therefore  $F[X] \models \psi(Q_0,P)$. Notice that the leading term in $Q_0$ is smaller (in the monomial ordering) then that one in $Q$.  Hence, by induction, $Q_0$ belongs to $F[P]$, so  does $Q$.
 
\end{proof}

\subsection{Interpretation of arithmetic in $F[X]$}

We start with the case when $F$ has characteristic zero. In this case $\Z$ is a subring of $F$, so it suffices to provide a formula $\phi(x)$ of the language of rings that defines $\Z$ in $F[X]$. 

\begin{lemma} \label{le:char0}
For any field $F$ of characteristic zero and any non-empty set $X$ the arithmetic $\N = \langle N\mid +,\cdot, 0,1\rangle \leq F$  viewed as a subset of $F$ is $0$-definable in $F[X]$ uniformly in  $F$  and $X$ (i.e., the defining  formula is the same for all fields $F$ of characteristic zero and all non-empty sets $X$).  
\end{lemma}
\begin{proof}
It was shown in \cite{Jensen}, Proposition 3.6, that $a \in F$ belongs to $\N$ if and only if it satisfies the following formula.

\begin{equation}
\forall u \not \in  F \exists v (u\mid v \wedge (\forall b \in F ((u+b) \mid v \rightarrow (u+b+1)\mid v \vee (b = a))))
\end{equation}
 
\end{proof}

  We show  below several results on interpretability of arithmetic  in the ring $F[X]$ for an arbitrary field $F$. The first part of the proof  (interpretability with parameters) is known (see, for example, \cite{Robinson} and \cite{Jensen}, Theorem 4.17). However, for the second and the third  we could not find any references.

\begin{lemma} \label{le:arbitrary-char}
Let $F$ be an arbitrary field and $X$ an arbitrary non-empty set.  Then the following hold:
\begin{itemize}
\item [1)] For any irreducible polynomial $a \in F[X]$ the arithmetic 
$\N = \langle N; +,\cdot, 0,1\rangle$ is interpretable with the parameter $a$  in $F[X]$ uniformly in  $F$, $X$,  and $a$ (i.e., the interpretation formulas are the same for all fields $F$, sets $X$, and irreducible polynomials $a$). We denote this interpretation by $\N_a$.
\item [2)] For any irreducible polynomials $a, b \in F[X]$ the canonical (unique)  isomorphism of interpretations $\mu_{a,b}:\N_a \to \N_b$ is definable in $F[X]$ uniformly in $F$, $X$,  and $a, b$.
\item [3)] The arithmetic $\N$ is 0-interpretable in $F[X]$.
\end{itemize}
\end{lemma}
\begin{proof}

Fix an arbitrary $a \in Irr$. Then the formula
$$
\phi_1(x,a) = \forall u (u \mid x \rightarrow (u \in F \vee a \mid u))
$$
defines in $F[X]$ a set $\{\alpha a^n \mid \alpha \in F, n \in \N\}$, while the formula 
$$
\phi_2(x,a) =  (a-1) \mid (x-1)
$$
defines in this set the subset 
$$
N_a = \{a^n \mid n \in \N\}.
$$
Hence the conjunction 
$$
Nat(x,a) = \phi_1(x,a) \wedge \phi_2(x,a)
$$
defines $N_a$ in $F[X]$. Clearly, for any $n,m,k \in \N$ one has

\begin{equation} \label{eq:operations-in-N}
n+m = k \Longleftrightarrow a^n\cdot a^m = a^k,
\end{equation}
\begin{equation} \label{eq:operations-in-N-2}
n\mid m \Longleftrightarrow (a^n-1)\mid (a^m-1).
\end{equation}

 The righthand sides of the equivalences above can be expressed by some first-order formulas of the ring theory (using the formula $Nat(x,a)$), say $\psi_+(a^n,a^m,a^k,a)$ and $\psi_\mid(a^n,a^m,a)$. This allows one define on the set $N_a$ a new structure, denoted $\N_a$,  which is isomorphic to the structure $\langle \N; +, \mid, 0 \rangle$ uniformly in the parameter $a \in Irr$. This proves 1).

To prove 2) we  show that for any $a,b \in Irr$ the isomorphism of the structures $\mu_{ab}:\N_a \to \N_b$, where $a^n \to b^n$ for $n \in \N$, is also definable  by a first-order formula with parameters $a,b$ uniformly in $F, X, a$ and $b$. For this we show first that the set 
$$
N_{ab} = \{(ab)^n \mid n \in \N\}
$$
is definable in $F[X]$ with parameters $a,b$. Indeed, the formula
$$
\forall u (u \mid x  \rightarrow [(u \not \in (Irr \cup F) \to ab \mid u) \wedge (u \in Irr) \to (a \mid u \vee b \mid u)]),
$$
which states that all no-irreducible non-invertible divisors of $x$  are divisible by $ab$, and all irreducible divisors of $x$  are divisible either by $a$ or by $b$, defines in $F[X]$ a subset 
$$
U = F\cdot N_{ab}\cdot \{a\} \cup F\cdot N_{ab}\cdot \{b\} \cup F\cdot N_{ab},
$$
(here and below for sets $M,K$ we denote $M\cdot K = \{mk \mid m \in M, k \in K\}$).  On the other hand, the sets 
$$
N_{a^2} = \{a^{2n} \mid n \in \N\}, \ \ \  N_{b^2} = \{b^{2n} \mid n \in \N\},
$$
are definable in $N_a$ and $N_b$, correspondingly. Hence they are definable in $F[X]$ (with parameters $a$ and $b$), as well as  the set
$$
V = N_{a^2} \cdot N_{b^2} = \{a^{2m}b^{2n} \mid m,n \in \N\}.
$$
It follows that the set
$$
W = U \cap V = \{(ab)^{2n} \mid n \in \N\} = N_{(ab)^2}
$$
is also definable in $F[X]$ with parameters $a,b$. Clearly, the set $N_{ab}$ can be expressed as 
$$
N_{ab} = \{(ab)^n \mid n \in \N\} = W \cup W\cdot\{ab\},
$$
so it is also definable in $F[X]$ with parameters $a, b$. 

Observe now that 
$$
\mu_{ab}(a^n) = b^m \Longleftrightarrow m = n \Longleftrightarrow a^nb^m \in N_{ab},
$$
hence there is a first-order formula $Is(x,y,a,b)$  which defines in $F[X]$ the map $\mu_{ab}$ uniformly in $a$ and $b$.

Now we interprete  the same structure $\langle \N; +, \mid, 0 \rangle$ in $F[X]$ without parameters.  Using the isomorphisms $\mu_{ab}$  one can glue all the elements $a^n$ for a fixed $n \in \N$ and $a$ running over $Irr$, into one equivalence class, by this identifying all the structures $\N_a$ into one structure  isomorphic to  $\langle \N; +, \mid, 0 \rangle$. The resulting structure is $0$-interpretable in $F[X]$ uniformly in $F$ and $X$, as claimed.

To finish the proof it suffices to notice that the standard arithmetic $\N = \{N; +,\times, 0,1\}$ is definable in the structure $\langle \N; +, \mid, 0 \rangle$ without parameters \cite{Robinson}.
\end{proof}

Now we improve on the result above allowing any non-invertible polynomial $P$ as a parameter (not only the irreducible ones).

\begin{lemma} \label{le:arbitrary-char-2}
Let $F$ be an arbitrary field and $X$ an arbitrary non-empty set.  Then the following hold:
\begin{itemize}
\item [1)] For any non-invertible  polynomial $P \in F[X]$ the arithmetic 
$\N = \langle N; +,\cdot, 0,1\rangle$ is interpretable with the parameter $P$  in $F[X]$ uniformly in  $F$, $X$,  and $P$. We denote this interpretation by $\N_P$.
\item [2)] For any non-invertible  polynomials $P, Q \in F[X]$ the canonical (unique)  isomorphism of interpretations $\mu_{P,Q}:\N_P \to \N_Q$ is definable in $F[X]$ uniformly in $F$, $X$,  and $P, Q$.
\end{itemize}
\end{lemma}
\begin{proof}
We use results and notation from Lemma \ref{le:arbitrary-char}. Let $a$ be a fixed irreducible polynomial in $F[X]$. Fix the interpretation  $\N_{a}$ of arithmetic and denote it by $\N$. By Lemma  \ref{le:arbitrary-char} for any $b \in Irr$ the map $a^m \to b^m, m \in \N$ is  definable uniformly in $a, b$. This allows us to use notation $b^m$, as well as $m \in \N$  in our formulas. Now we follow the scheme of the proof  in Lemma \ref{le:arbitrary-char}. 

Observe that the formula 
$$
\phi_1(x,P,m,a) = \forall b \in Irr [(b\mid P \to (b^m \mid x) \wedge  \neg(p^{m+1}\mid x)) \wedge (b\mid x \to b\mid P)]
$$
which states that $x$ and $P$ have precisely the same irreducible divisors, and every irreducible divisor of $P$ occurs in $x$ precisely $m$ times, defines in $F[X]$ the set $\{\alpha P^m  \mid \alpha \in F\}$. Hence, the formula 
$$
\phi_2(x,P,m,a) = \phi_1(x,P,m,a)  \wedge (P-1)\mid (x-1)
$$
defines in $F[X]$ the element $P^m$. Therefore, the formula 
$$
\phi_3(x,P) = \exists a \in Irr \exists m \in \N_a\phi_2(x,P,m,a) 
$$
defines in $F[X]$ the set 
$$
\N_P = \{P^m \mid m \in \N\}.
$$
As in Lemma \ref{le:arbitrary-char} (see conditions (\ref{eq:operations-in-N}) and (\ref{eq:operations-in-N-2})), for any $n,m,k \in \N$ one has
$$
n+m = k \Longleftrightarrow P^n\cdot P^m = P^k,
$$
$$
n\mid m \Longleftrightarrow (P^n-1)\mid (P^m-1).
$$
Hence there are formulas $\psi_+(P^n,P^n,P^k,P)$ and $\psi_\mid(P^n,P^n,P^k,P)$ that define the addition + and the division $\mid$ on $\N_P$. So the arithmetic is interpretable on $\N_P$ uniformly in $F, X, P$, as claimed in 1). 
 
 To see 2) observe that by construction the formula $\phi_2(x,P,m,a)$ gives the canonical isomorphism $\N_{a} \to \N_P$ defined by $a^m \to P^m$. Hence for a non-invertible $Q \in F[X]$ the formula 
 $$
 \exists a \in Irr \exists m \in \N_a \phi_2(x,P,m,a) \wedge \phi_2(y,Q,m,a)
 $$
 defines the canonical isomorphism $\mu_{P,Q}:\N_P \to \N_Q$ of the interpretations $\N_P$ and $\N_Q$, as required.
 
\end{proof}
 
 Now we give one more interpretation of $\N$ in $F[X]$ and show that it is definably isomorphic with the previous ones.
 
 In the notation of Lemma \ref{le:int-poly-one}  the one-variable ring of polynomials $F[P]$ is definable in $F[X]$ uniformly in $F, X$ and $P$. Since $P$ is irreducible in $F[P]$ by Lemma \ref{le:arbitrary-char} the arithmetic $\N$ is interpretable in $F[P]$ (hence in $F[X]$) uniformly in $F, X$ and $P$. Denote this interpretation by $\N_P^\prime$.
 
 \begin{lemma} \label{le:inter-iso-3}
 Let $F$ be an arbitrary field and $X$ an arbitrary non-empty set. Then for any non-invertible polynomial $P \in F[X]$ the interpretation 
 $\N_P^\prime$ (see above) and the interpretation $\N_P$ from Lemma \ref{le:arbitrary-char-2} are definably isomorphic uniformly in $P$.
 \end{lemma}
 \begin{proof}
 By inspection of the arguments in Lemmas \ref{le:int-poly-one} and \ref{le:arbitrary-char-2} one can see that these interpretations have the same base set, namely $\N_P = \{P^m \mid m \in \N\}$ (though defined by different formulas) and  precisely the same operations  given by formulas
  (\ref{eq:operations-in-N}) and (\ref{eq:operations-in-N-2}). The formula $\phi_1(x,P,m,a) $ from Lemma \ref{le:arbitrary-char-2} defines the isomorphism between the interpretations $\N_P^\prime$ and $\N_P$.
 \end{proof}

For a field $F$ of characteristic zero by $\N_1$ we denote the interpretation of the arithmetic $\N$ in $F[X]$ as a subset of $F$ from Lemma \ref{le:char0},  and by $\N_2$  - the interpretation from Lemma \ref{le:arbitrary-char}. The following result shows that we can use any of these  interpretations as we pleased. 
\begin{lemma}
The canonical isomorphism $\lambda: \N_1 \to \N_2$ is definable in $F[X]$.
\end{lemma}
\begin{proof}
In the  notation from Lemmas \ref{le:char0} and \ref{le:arbitrary-char} one needs to construct a formula $\Delta(x,y,z)$ such that for elements $b \in Irr$, $v \in F[X]$,  and $m \in \N_1 \leq F$ one has $F[X] \models \Delta(v,m,b)$ if and only if $v = b^m$. 

Let $a\in Irr$ be such that $a+1 \in Irr$, for example  $a$ could be any polynomial of degree 1 in $F[X]$.

Let $\mu_{a,a+1} : a^\N \to (a+1)^\N$ be the definable isomorphism  from Lemma \ref{le:arbitrary-char} such that $a^m \to (a+1)^m$ for $m \in \N$. Hence there is a formula $\Delta_1(x,y,z)$ such that for any $u,v \in F[X]$ 
$$
F[X] \models \Delta_1(u,v,a) \Longleftrightarrow \exists m \in \N (u = a^m \wedge  v = (a+1)^m).
$$

By the binomial formula 
$$
(a+1)^m = a^m +ma^{m-1} + \ldots +ma +1,
$$
hence  for  $m \in \N, m \neq 0,$  there exists a unique $w \in F[X]$ such that 
$$
(a+1)^m = a(aw+m) +1.
$$
Note that this condition can be written by a formula, say $\Delta_2(a,m)$.

It follows that the formula 
$$
\Delta_3(u,m,a) = \exists v (\Delta_1(u,v,a) \wedge \exists w(v = a(aw+m) +1) 
$$
defines the isomorphism $m \to a^m$ from $N_1$ to $N_a$. Now let $b$ be an arbitrary element in $Irr$. 
The isomorphism $\mu_{a,b}: N_a \to N_b$ is definable uniformly in $F[X]$ by a formula $Is(x,y,a,b)$ from Lemma \ref{le:arbitrary-char}, hence the formula 
$$
\Delta(v,m,b) = \exists a \exists u [(a\in Irr) \wedge (a+1 \in Irr) \wedge (\Delta_3(u,m,a) \wedge v = \mu_{a,b} (u) )]
$$
gives the required isomorphism $\lambda: \N_1 \to \N_2$.
\end{proof}

\subsection{Interpretation of the weak second order theory of $F$ in $F[X]$}

Following ideas of Bauval \cite{bauval} we prove the following result. Notice, that uniform interpretability and definability of the  isomorphisms  of the interpretations seem to be unknown before.
 \begin{theorem}  \label{Bauva}
 Let  $F$ be an infinite field and $X$ an arbitrary non-empty set. Then the following hold:
 \begin{itemize}
\item [1)]  for a given non-invertible polynomial $P \in F[X]$ one can interpret $S(F,\N)$ in $F[X]$ using the parameter $P$ uniformly in $F$, $X$, and $P$. We denote this interpretation by $S(F,\N)_P$.
\item [2)] for any non-invertible polynomials $P, Q \in F[X]$ the canonical (unique)  isomorphism of interpretations $\nu_{P,Q} : S(F,\N)_P \to S(F,\N)_Q$ is definable in $F[X]$ uniformly in $F$, $X$, $P$, and $Q$.
\item [3)]  $S(F,\N)$ is 0-interpretable  in $F[X]$  uniformly in $F$ and $X$.
 
 \end{itemize}
 \end{theorem}
 \begin{proof}
   By Lemma \ref{le:int-poly-one} for  a non-invertible polynomial $P \in F[X]$ the polynomial ring $F[P]$ is definable in $F[X]$ with parameter $P$ uniformly in $F, X$ and $P$. So it suffices to show that the structure $S(F,\N)$ is interpretable in a ring of polynomials in one variable, say $F[t]$,  with the variable $t$ in the language, uniformly in $F$. To this end consider the language of ring theory $L_t$ with the element $t$  as a new constant.   By  Lemma \ref{le:arbitrary-char} the arithmetic $\N_t$ is interpretable in $F[t]$ in the language $L_t$ uniformly in $F$. So the set $N_t = \{t^n \mid n \in \N\}$, as well as the addition and the multiplication in $\N_t$, is definable in $F[t]$ by a formula with the parameter $t$.  This gives a required interpretation in $F[t]$ of the third sort $\N$ of the structure
 $$
 S(F,\MN) = \langle F, S(F),\MN; t(s,i,a), l(s), \frown\rangle.
 $$
Now we interpret $S(F)$ in $F[t]$. We associate a sequence $\bar \alpha = (\alpha_0, \ldots, \alpha_{n})$ of elements from $F$ with a pair $s_{\bar \alpha} = (\Sigma_{i=0}^n \alpha_it^i,t^n)$. We need to show that the set of such pairs is definable in $F[t]$ by a formula in $L_t$. Observe that a polynomial $f(t) \in F[t]$ has degree at most $n$ if and only if a rational function $t^n f(\frac{1}{t})$ is again a  polynomial from $F[t]$. This leads to the following formula:
$$
\phi(f,t,t^n) = \exists g \forall \alpha \in F \smallsetminus \{0\} \exists \beta, \gamma \in F \Big (( t-\frac{1}{\alpha} \mid f -\beta) \wedge (t-\alpha \mid t^n -\gamma) \wedge (t-\alpha \mid g-\beta\gamma) \Big ).
$$
Note, that $t-\frac{1}{\alpha} \mid f -\beta$ gives  $f(\frac{1}{\alpha}) = \beta$, similarly $t-\alpha \mid t^n -\gamma $ is equivalent to $\alpha^n = \gamma$, and $t-\alpha \mid g-\beta\gamma $ means $g(\alpha) = \beta\gamma$. 

Combining these conditions together, one gets that if $F[t] \models \phi(f,t,t^n)$, then $g(\alpha) = f(\frac{1}{\alpha}) \alpha^n$ for infinitely many $\alpha$ (since the field $F$ is infinite). Hence  $g(t) = f(\frac{1}{t}) t^n$, as required. It follows that the formula $\phi(x,t,y)\wedge (y \in \N_t)$ defines in $F[t]$ precisely the set of pairs 
$$
\{(x,y) \mid x =  \Sigma_{i=1}^n \alpha_it^i, y = t^n (n \in \N)\}.
$$
 This gives a 0-interpretation  in $F[t]$ (viewed in the language $L_t$) of the set  $S(F)$ of  all tuples of $F$. Note, that the field $F$ is also 0-interpretable in $F[t]$, so the two sorts of the structure $S(\MF) = \langle F, S(F), \frown,\in\rangle$ are 0-interpretable in $F[t]$ in the language $L_t$. 
 To finish the proof of 1) one needs to show that the operations $t(s,i), l(s)$, and $\frown$ are also 0-interpretable in $F[t]$ in the language $L_t$ (recall that that in this case, as was mentioned in Section \ref{se:weak-second-order-intro},  the predicate $\in$ is also 0-interpretable in $F[t]$).
 
 Let $\bar \alpha = (\alpha_0, \ldots, \alpha_{n}), \bar \beta = (\beta_0, \ldots,\beta_m)$ be two sequences of elements from $F$, $s_{\bar \alpha} = (\Sigma_{i=0}^n \alpha_it^i,t^n) = (f,t^n)$ and $s_{\bar \beta} = (\Sigma_{i=0}^m \beta_it^i,t^m) = (g,t^m)$ their interpretations in $F[t]$. Then the sequence $\bar \alpha \frown \bar \beta$ obtained by concatenation from $\bar \alpha$ and $\bar \beta$ corresponds to the pair $(f+t^{n+1}g,t^n\cdot t^m)$, so the operation of concatenation is 0-definable in $F[t]$ in the language $L_t$.
 
 The length function $\ell: (\alpha_0, \ldots, \alpha_{n}) \to n+1$ is also 0-definable in $F[t]$ in the language $L_t$. Indeed, the length of the pair $(\Sigma_{i=0}^n \alpha_it^i,t^n) = (f,t^n)$ is precisely $t^{n+1} = t^n\cdot t \in N_t$. 
 
 Using operations $\frown$ and $\ell$ one can define the predicate $t(s,i,a)$ in $F[t]$ as follows. Conditions 
 \begin{itemize}
 \item there are sequences $s_1, s_2, s_3$ such that $s = s_1\frown s_2 \frown s_3$;
 \item $\ell(s) = n+1$, $\ell(s_1) = i$, $\ell(s_2) = 1$, and $\ell(s_3) = n-i$;
 \item $s_2 = (\alpha,t^0)$ and $a = \alpha$.
 \end{itemize}
  are 0-definable in $F[t]$ in the language $L_t$ and their conjunction defines the predicate $t(s,i,a)$. 
  
  We showed that for a given non-invertible polynomial $P \in F[X]$ one can interpret $S(F,\N)$ in $F[X]$ using the parameter $P$ uniformly in $F$, $X$, and $P$. We denote this interpretation by 
  $$
  S(F,\N)_P = \langle F, S(F)_P, \N_P, t_P(s,i,a), l_P(s), \in _P \rangle.
  $$
   This proves 1).
 
 Now we show that  for different non-invertible parameters $P_1, P_2 \in F[X]$ there is a uniformly definable isomorphism 
 $$
 \nu_{P_1,P_2}:  S(F,\N)_{P_1} \to S(F,\N)_{P_2}.
 $$
    Observe, that the interpretation of the first sort $F$ in $S(F,\N)_P$ does not depend on $P$. The definable isomorphism $\mu_{P_1,P_2}:\N_{P_1} \to \N_{P_2}$ between the third  sorts in $S(F,\N)_{P_1} $ and $S(F,\N)_{P_2} $ was constructed in Lemma \ref{le:arbitrary-char-2} (see also  Lemma \ref{le:inter-iso-3}). 
    
    Now it is suffices  to show that the  isomorphism $\sigma_{P_1,P_2}: S(F)_{P_1} \to S(F)_{P_2}$ between the second  sorts $S(F)_{P_1}$ and $S(F)_{P_2}$ in $S(F,\N)_{P_1} $ and $S(F,\N)_{P_2} $ which arises  from the identical map $S(F) \to S(F)$ is  definable in $F[X]$ uniformly in $F,X$, $P_1$, and $P_2$. Indeed,  if $s_{\bar \alpha} = (f,P_1^n) \in S(F)_{P_1}$ and $s_{\bar \beta} = (g,P_2^m) \in S(F)_{P_2}$ then for such $\sigma_{P_1,P_2}$ one has $\sigma_{P_1,P_2}(f,P_1^n) = (g,P_2^m) $ if and only if $n = m$ and for each $0\leq i \leq n$ the $i$'s components of the tuples $\bar \alpha$ and $\bar \beta$ are equal. The latter  means that for each  $a,b \in F$ such that $t_{P_1}(s_{\bar \alpha},i,a)$ and $t_{P_2}(s_{\bar \beta},i,b)$ hold in $F[X]$ one has $a = b$.  All these conditions can be written by formulas of the ring theory uniformly in $F, X, P_1, P_2$.  This proves 2).

3) follows from 2) by an argument similar to the one in Lemma \ref{le:arbitrary-char}.

This finishes the proof.

 \end{proof}

\section{Tarski problems for $F[X]$}

By Lemma \ref{le:arbitrary-char} the arithmetic $\N$ is interpretable in $F[X]$, 
as a corollary one  gets the following known result due to R.Robinson. 
\begin{theorem}\cite{Robinson}\label{Rob}
For any field $F$ and any non-empty set $X$ the first-order theory of $F[X]$ is undecidable. 
\end{theorem}

The following result characterises first-order equivalence of rings of polynomials over arbitrary fields.

\begin{theorem} \cite{bauval} Let $F$ be a  field and $X$ a finite non-empty set. Then for any field $K$ and any set $Y$ one has $F[X] \equiv K[Y]$ if and only if 
 $|X|=|Y|$ and $HF(F)\equiv HF(K).$\end{theorem}

\begin{proof}
Suppose $F[X]\equiv K[Y]$. Then they have the same (finite)  Krull dimension, therefore $|Y|=|X|.$  By Lemma \ref{le:field}  $F \equiv K$. If one of the fields is finite then the other one is and in this case they are isomorphic, in particular, $HF(F)\equiv HF(K).$

 If the fields are infinite then by Theorem \ref{Bauva}  
 the model $S(F,\N)$, hence the model $HF(F)$,  is 0-interpretable in  the ring $F[X]$ uniformly in $F$ and $X$. Therefore $HF(F)$  and $HF(K)$ are 0-interpretable in $F[X]$ and $K[Y]$ by the same formulas of the ring language.  By Corollary \ref{co:interp}     $F[X]\equiv K[Y]$ implies $HF(F)\equiv HF(K).$

 Conversely, suppose  $|X|=|Y| \le \infty$ and $HF(F)\equiv HF(K).$ One  can enumerate all the monomials in $F[X]$ and 0-interprete the free monoid on $X$ in $\N$. Now one can represent each element in $F[X]$  as a finite sequence  of coefficients in $F$. Addition and multiplication in $F[X]$ is interpretable in the weak second order logic of $F$. Therefore $F[X]$  is 0-interpretable in  $HF(F)$ uniformly on $F$ and $|X|$.  

\end{proof}

\begin{cor} If $F$ is one of the fields from Section 2.3, then the polynomial rings $F[X]$ and $K[Y]$  are elementarily equivalent if and only if  they are isomorphic.\end{cor}

The following theorem describes finitely generated rings (or Noetherian) rings first-order equivalent to $F[X]$.

\begin{theorem} \cite{bauval} \label{th:Bauval_poly}
A noetherian ring $R$ is first-order equivalent to $F[X]$  if and only if it is  isomorphic to a polynomial ring $K[Y]$ where $|X|=|Y|$ and $HF(F)\equiv HF(K).$
\end{theorem}

\section{Interpretability  in $\MA_K(X)$} \label{se:Interpet_free assoc}

In the rest of the paper let $K$ be a field, $X =\{x_1, x_2, \ldots, \}$ a set, and $\MA = \MA_K(X)$ a free associative unitary algebra with basis $X$ and coefficients in $K$. By   $K[t]$ we denote a polynomial ring in one variable $t$ with coefficients in $K$. By $X^\ast$ or $\MM_X$ 
we denote the free monoid with basis $X$ viewed as a set of all words in the alphabet $X$. We identify $\MM_X$ with the set of all monomials in $\MA_K(X)$ with respect to the fixed basis $X$, so we  refer to elements in $\MM_X$ either as to words in $X$ or monomials in $X$. Let $L$ be the standard language of rings with identity $1$, consisting of the binary operations operations $+,\cdot$ and the constant symbol 1.
By $L_X$ we denote the language   which is obtained from $L$  by adding the  elements from the set $X $   as new constants.  If not mentioned precisely otherwise we assumed that all the formulas that occur are in the language $L$.

\subsection{Basic facts}

The following result is crucial for our considerations, it allows one to transfer some principal results on definability in $K[t]$ into $\MA_K(X)$ ($X \neq \emptyset$).

\begin{theorem} [Bergman, \cite{Bergman}] \label{th:Bergman}
The centralizer in $\MA_K(X)$ of a non-invertible  polynomial is isomorphic to the polynomial ring $K[t]$ in one variable $t$ with coefficients in $K$.
\end{theorem}

\begin{cor} \label{co:Bergman} 
 Let  $K$ be a  field and $X$ an arbitrary non-empty set. For any  non-invertible polynomial  $P \in \MA_K(X)$ one can interprete the ring of polynomials $K[t]$  in $\MA_K(X)$ using  the parameter $P$ as the centralizer $C_{\MA_K(X)}(P)$   uniformly in $K, X$ and  $P$. 
\end{cor}

 \begin{theorem} \label{th:rank-non-com}
 For any natural number $n \in \N$ there exists a set of first-order sentences $\Psi_n$ of the ring theory language $L$ such that for any field $K$ and any set $X$ the set of sentences $\Psi_n$ holds in $\MA_K(X)$ if and only if $|X| = n$. 
 \end{theorem}
 \begin{proof}
By definition the rank of  $\MA_K(X)$ is zero if and only if $X =  \emptyset$, i.e., $\MA_K(X) = K$, so precisely when every non-zero element in $\MA_K(X)$  is invertible. This condition can be described by a sentence, its singleton set  gives  $\Psi_0$.  The case  of $|X| = 1$ is also easy, since it is suffices to write down that $\MA_K(X)$  is commutative, but not a field.  This gives $\Psi_1$.
 
 Now assume that $|X| = n$ and $\MA = \MA_K(X)$. Then  every element $a \in \MA_K(X)$ has a unique decomposition of the form
  \begin{equation} \label{eq:auto}
 a = x_1a_1 + \ldots + x_na_n+\alpha,  \ \ \ a_i \in \MA, \alpha \in K.
  \end{equation}
 Indeed, let $a = \alpha_1w_1 + \ldots + \alpha_kw_k +\alpha_0\cdot 1$, where $\alpha_i \in K, w_i \in \MM_X$ be the unique decomposition of $a$ via monomials from  $\MM_X$ of $\MA$. Collecting all terms $\alpha_iw_i$ such that $w_i$ begins with $x_1$ and factoring  $x_1$ out  to the left one gets the element $a_1$. Now collecting for $x_2$ in the element $a-x_1a_1$ one gets $a_2$, and so on.
 Existence and uniqueness of the decomposition (\ref{eq:auto}) for any element in $\MA$ can be described  by a formula in $L$ with parameters $x_1, \ldots,x_n$. More precisely, consider the following  formula in the language $L$:
$$
\phi_{1,m}(y_1, \ldots,y_m)=\forall a\in \MA\exists ! a_1,\ldots ,a_m\in \MA\exists !\alpha\in K (a=\sum y_ia_i+\alpha).
$$
As we showed above the formula $\phi_{1,n}(y_1, \ldots,y_n)$ holds in $\MA_K(X)$ on the elements $x_1, \ldots,x_n$.  
Notice also, that the formula $\phi_{1,m}(y_1, \ldots,y_m)$ does not hold in $\MA_K(X)$ on any tuple of elements $b_1, \ldots, b_m$ provided the basis $X$ is infinite. Indeed, there are only finitely many elements from $X$ that occur in monomials of these elements, so any element $a$  from $X$ that does not occur in these monomials cannot be represented in the form $a = b_1a_1 + \ldots + b_na_n+\alpha$ above.

Observe that the definable subset 

\begin{equation} \label{eq:ideal-2 }
  I_X = x_1\MA+ \ldots + x_n\MA
  \end{equation}
is a two-sided  ideal in $\MA$,  and this  can be described  by a formula, say $\phi_{2,n}(y_1, \ldots,y_n)$ in the language $L$, which states that  for any tuple $B = (b_1, \ldots, b_n)$ over $\MA$ the definable set $I_B = b_1\MA + \ldots + b_n\MA$ is a two-sided ideal in $\MA$.

Now,    $\MA$ admits, as a vector space over $K$,  a direct decomposition 
\begin{equation} \label{eq:ideal }
 \MA = K \oplus I_X.
  \end{equation}
The definable with parameters $x_1, \ldots,x_n$ set 
$$
\sum_{1\leq i,j \leq n} x_ix_j\MA,
$$
 is a two-sided ideal  in $\MA$, moreover this ideal is equal to $I_X^2$ - the square of $I_X$. This, again, can be described by a formula, say $\phi_{3,n}(y_1, \ldots,y_n)$, which states that  for any tuple $B = (b_1, \ldots, b_n)$ over $\MA$ the definable set $ \sum_{1\leq i,j \leq n} b_ib_j\MA $ is a two-sided ideal in $\MA$ and this ideal is equal to $I_B^2$. 

Clearly,  $\MA/I_X^2$ has dimension $n+1$ over $K$, which can be described by the following  formula:
$$
\phi_{4,n}(y_1, \ldots,y_n) = \forall a  \exists \alpha_0 \in K \ldots  \exists \alpha_n \in K \exists b \in I_Y^2 (a = \sum_i \alpha_iy_i +\alpha_0+b).
$$
Put $\phi_n(y_1, \ldots,y_n) = \phi_{1,n} \wedge \phi_{2,n} \wedge \phi_{3,n} \wedge \phi_{4,n}$.
By construction $\phi_n$ holds in $\MA$ on $(x_1, \ldots,x_n)$,

Now suppose that  $\phi_m$ holds in $\MA$ on elements  $b_1, \ldots,b_m$ , so  $\MA \models \phi_m(b_1, \ldots,b_m)$. Then the set  $I_B = b_1\MA + \ldots + b_m\MA$  is a two-sided ideal in $\MA$ (because $\MA \models \phi_{2,m}(b_1, \ldots,b_m)$). Since $\MA \models \phi_{3,m}(b_1, \ldots,b_m)$ one has $\MA = I_B \oplus K$.  For any $a \in \MA$ there exist unique $a' \in I_B$ and $\alpha(a) \in K$  such that 
\begin{equation} \label{eq:aprime}
a = a'+\alpha(a).
\end{equation}
 The map $h_B: \MA \to I_B$ such that $ a \to a'$  is definable in $\MA$ with parameters $B$ uniformly in $B$ satisfying $\phi_m$. For any $a, b \in \MA$ one has $a+b = a'+\alpha(a) + b'+\alpha(b)$, so $(a+b)' = a'+b'$ (from uniqueness of the decomposition (\ref{eq:aprime})). Similarly, $(\alpha a)' = \alpha a'$ for any $\alpha \in K$. It follows that $h_B$ is a $K$-linear. 
 
 
 We claim that the set $h_B(X)$  generates $K$-vector space $\MA$ modulo $K + I_B^2$, i.e., $\MA = \langle x_1', \ldots, x_n'\rangle_K + K +I_B^2$ as a vector space. Indeed, observe that for any $a, b \in \MA$ one has $ab = a'b' +a'\alpha(b) + \alpha(a)b' + \alpha(a)\alpha(b)$, so  $h_B(ab) \in \langle a',b', 1\rangle_K + I_B^2$. Similarly, for any $a_1, \ldots, a_t \in \MA$ one has $h_B(a_1 \ldots a_t) \in \langle a_1', \ldots, a_t',1\rangle_K + I_B^2$. Since  $X$ generates $\MA$ as an algebra it follows that $\MA = \langle h_B(X) \rangle_K+K  +I_B^2$, as claimed.  We showed that  for any $B = (b_1, \ldots,b_m)$ satisfying $\phi_m$ in $\MA$ one has 
 \begin{equation}  \label{eq:h-hom}
 \langle h_B(X)\rangle_K +K +I_B^2 = \MA.
 \end{equation} 
 Since $h_B$ is definable with parameters $B$ and the action of $K$ on $\MA$ is also definable it follows that the condition \ref{eq:h-hom} can be written by a formula, say $\phi_{5,n,m}(y_1, \ldots,y_n)$. Therefore the following formula holds in $\MA$ on elements $x_1, \ldots,x_n$:
 $$
 \phi_{n,m}(y_1, \ldots,y_n) = \phi_n(y_1, \ldots,y_n) \wedge \forall b_1 \ldots b_m(\phi_m(b_1, \ldots,b_m) \to \phi_{5,n,m}(y_1, \ldots,y_n)).
 $$
Now if $C= (c_1, \ldots,c_n)$ satisfies $\phi_{n,m}(y_1, \ldots,y_n)$ in $\MA$ and $D = (d_1, \ldots,d_m)$ satisfies $\phi_m(y_1, \ldots,y_m)$ in $\MA$ then $c_1, \ldots,c_n$ generates $\MA$ modulo $K +I_D^2$ and $d_1, \ldots, d_m,1$ is a $K$-basis of $\MA/I_D^2$. Hence $n \geq m$. Notice that $(x_1, \ldots,x_n)$ satisfies $\phi_{n,m}(y_1, \ldots,y_n)$ in $\MA$, so the sentence 
$$
\psi_{n,m} = \exists y_1 \ldots y_n \phi_{n,m}(y_1, \ldots,y_n)
$$
holds in $\MA$ for any $m$. Set $\Psi_n = \{\psi_{n,m} \mid m \in \mathbb{N}\}$. We showed that $\MA \models \Psi_n$.  

Suppose that $\MA \models \Psi_t$ for some $t \in \mathbb{N}$. Then the sentences $\psi_{n,t}$ and $\psi_{t,n}$ both hold in $\MA$ hence $n = t$. This shows that $\MA \models \Psi_t$ if and only if $t = n$.

Suppose now that $\MA = \MA_K(X)$ with infinite set $X$. If  $\MA \models \Psi_t$ for some $t \in \mathbb{N}$ then by construction $\MA \models \exists y_1 \ldots y_t \phi_t$. Hence, as was mentioned above the set $X$ must be finite. This shows that $\Psi_t$ does not hold in $\MA_K(X)$  for any $t \in \mathbb{N}$. 
\end{proof}

\begin{lemma} \label{le:monomials}
For a finite $X$ the following holds:
\begin{itemize}
\item [1)] the  monoid $K\MM_X = \{ \alpha w \mid \alpha \in K, w \in \MM_X\}$ is definable in $\MA_K(X)$ with parameters  from $X$ uniformly in $K$ and the cardinality $|X|$.

\item [2)] the monoid $\MM_X$ is interpretable  in $\MA_K(X)$ with parameters  from $X$ uniformly in $K$ and the cardinality $|X|$.
\end{itemize}
\end{lemma}
\begin{proof}
An element $a \in  \MA_K(X)$ belongs to  $K\MM_X$ if and only if it satisfies  the condition that every non-invertible divisor of $a$ is divisible by one of the elements from $X$, so 
the following formula $\phi(a,X)$ in $L_X$  defines $K\MM_X$ in $\MA_K(X)$:
$$
\phi(a,X) = \forall b  (b\mid a \to (b \in K)\vee \bigvee_{i = 1}^n (x_i \mid b)).
$$

This proves 1).  To see  2) notice first that an  equivalence relation $\sim$ on $ \MA_K(X)$ defined by $x \sim y \Longleftrightarrow \exists \alpha \exists \beta  (x = \alpha y) \wedge (\alpha\beta = 1)$ is definable in $ \MA_K(X)$, hence the quotient monoid $K\MM_X/\sim$, which is isomorphic to $\MM_X$,  is interpretable in $ \MA_K(X)$ uniformly in $K$ and $|X|$.
\end{proof}

 \subsection{Interpretation of arithmetic $\N$ in $\MA_K(X)$} \label{se:N-in-A}
In this section  $K$ is  an arbitrary field and $X$ is a set with $|X| \geq 2$. 

 By Corollary   \ref{co:Bergman} for a non-invertible polynomial  $P \in \MA$ the one-variable polynomial  ring  $K[t]$ is  definable in $\MA = \MA_K(X)$ as the centralizer $C_\MA(P)$ with the parameter $P$     uniformly in $K,X$ and $P$. Notice that $P$ could be reducible in $K[t]$.   However, by Lemma \ref{le:int-poly-one} the ring $K[P]$ is definable with the parameter $P$ in $K[t]$ uniformly in $K$ and $P$, hence $K[P]$ is definable with the parameter $P$ in $\MA_K(X)$  uniformly in $K$, $X$,  and $P$. By Lemma \ref{le:arbitrary-char} the arithmetic $\N$  is interpretable in $K[t]$ with an arbitrary non-invertible parameter $P$ on the set of all powers $\{ P^m \mid m \in \N\}$ uniformly in $K$ and $P$, where the addition and multiplication for $n,m,k \in \N$  is defined by 
 $$
n+m = k \Longleftrightarrow P^n\cdot P^m = P^k,
$$
$$
n\mid m \Longleftrightarrow (P^n-1)\mid (P^m-1). 
$$
As in Lemma \ref{le:arbitrary-char} we denote this interpretation by $\N_P$. In particular, for an arbitrary  non-invertible polynomial  $P \in \MA$  one has interpretation $\N_P$ uniformly in $K,X$, and $P$.  The main result in this section is that   the interpretations $\N_P$ are definably isomorphic in $\MA$. 

To this end we  introduce a particular form of a useful technique of "big powers".

Given a number $m \in \MN$ we define an element $a_m \in  \MM_{\{x_1,x_2\}} $ by
\begin{equation} \label{eq:a}
a_m = x_1x_2x_1x_2^2x_1x_2^3\ldots x_1x_2^m.
\end{equation}

\begin{lemma} \label{le:a-m}
The set of pairs  $B = \{(\alpha a_m,m) \mid \alpha \in K,  m \in \N_{x_2}\}$ is definable in $\MA$ uniformly in $K$ and $x_1, x_2 \in X$ such that $x_1 \neq x_2$.
\end{lemma}
\begin{proof}
The monomials  $\alpha a_m$ are  completely determined by the number $m$ and the  following conditions:
\begin{itemize}
\item [1)] (divisors)  $x_1$ and $x_2$ are the only irreducible  divisors of $a_m$;
\item [2)] (endpoints) $a_m = x_1x_2x_1wx_1x_2^m$ for some $w \in K\MM_{\{x_1,x_2\}}$;
\item [3)] (recursion) if $a_m = w_1x_1x_2^jx_1w_2$ for some $w_1, w_2 \in K\MM_{\{x_1,x_2\}}$ and  $j <m$ then $w_2 = x_2^{j+1}x_1w_3$ for some $w_3 \in K\MM_{\{x_1,x_2\}}$.
\item [4)] (uniqness) if $a_m = w_1x_1x_2^jx_1w_2 = w_1^\prime x_1x_2^jx_1w_2^\prime$ for some $w_1, w_2, w_1^\prime, w_2^\prime \in K\MM_{\{x_1,x_2\}}$ and  $j <m$ then $w_1  \sim w_1^\prime, w_2 \sim w_2^\prime$, where $\sim$ is the equivalence relation on $\MA$ such that $x \sim y$ provided $x = \alpha y $ for some $\alpha \in K$.
\end{itemize}
Observe that these conditions are definable in $\MA$ in the language $L_{\{x_1,x_2\}}$. Indeed, we can define by formulas  the condition $u = x_2^j$ since $\N_{x_2}$ is interpretable  in $\MA$. The predicate $v \in K\MM_{\{x_1,x_2\}}$ is also definable in $\MA$ in the language $L_{\{x_1,x_2\}}$ by Lemma \ref{le:monomials}. This proves the lemma.

\end{proof}

 \begin{lemma} \label{le:main-decomp} Let $f_1,\ldots, f_{s+1}$ be non-invertible polynomials in ${\mathbb A}_K(X)$. Suppose $a\in \MM_X$ is a monomial such that:  $a$ is not a proper power, $a \neq a_1a_2a_1$ for any non-trivial $a_1,a_2 \in \MM_X$,   and $a$ is not contained as a subword in any of the monomials in $f_1,\ldots, f_{k+1}$.  Fix an integer $e \geq 3$ and consider a polynomial $f \in \MA$ defined as 
\begin{equation}\label{eq:decamp-a} 
f=f_0a^ef_1a^{e+1}\ldots a^{e+s}f_{s+1}
\end{equation} 
Then the following holds:
\begin{itemize}
\item [1)] Each maximal occurrence of $a^{j}$ in $f$ as a multiplicative factor  is uniquely defined up to a constant, i.e., if 
$$f = g_1a^{j}g_2 = g_1^\prime a^{j}g_2^\prime,$$
 where $g_1,g_1^\prime  \not \in \MA a$ and $g_2,g_2^\prime  \not \in a\MA$ (this defines a maximal occurrence) then 
 $$g_1 = \alpha_1 g_1^\prime, \ \ \ g_2 =   \alpha_2 g_2^\prime ,$$ for some  $\alpha_1, \alpha_2 \in K, \alpha \neq 0$ (in which case $\alpha_1 = \alpha_2^{-1}$).
\item [2)] Each occurrence of $a^{e+i}$ in  (\ref{eq:decamp-a}) is maximal, and there are no any other maximal occurrences of the type $a^j$ in $f$.
\item  [3)] The decomposition (\ref{eq:decamp-a})  is a unique  $a$-decomposition of $f$, i.e., if  
$$f=f_1^\prime a^ef_2^\prime a^{e+1}\ldots a^{e+s}f_{s+1}^\prime$$
  is another such  a decomposition of $f$  then $f_i = \alpha_i f_i^\prime$ for some $\alpha_i \in K$,  $i = 1, \ldots,s+1$. 
  
  \item [4)] If $f_1,\ldots, f_{s+1}$ are such that their leading  monomials in the shortlex order occur with coefficient 1 then in the condition 3) above one has $f_i = f_i^\prime$ for $i = 1, \ldots,s+1$.
\end{itemize}
\end{lemma}
\begin{proof}
We will first show 1).  Suppose we have a maximal occurrence $$f = g_1a^{j}g_2 = g_1^\prime a^{j}g_2^\prime,$$ as in 1). If $g_i = \alpha_iM_i, g_i' = \alpha_i'M_i'$, where $\alpha_i, \alpha_i' \in K$ and $M_i, M_i'$  are monomials, then $M_1a^jM_2 = M_1'a^jM_2'$ and  the statement 1) is true, because it is true  in the  free semigroup generated by $X$.  In the general case,  we use induction on the width (the number of monomials which occur with non-zero coefficients) of the polynomials  $g_i$ and $g_i'$ (i=1,2). Denote by $M_i$ and $M_i'$ the leading monomials in the shortlex order of the polynomials $g_i$ and $g_i'$, correspondingly (i=1,2),  so $g_i = \alpha_iM_i +h_i$ and $g_i' = \alpha_i'M_i' +h_i'$,  where  $\alpha_i, \alpha_i' \in K$  and the polynomials $h_i$ and $h_i'$  have  smaller width than $g_i$ and $g_i'$. Now, the equality $g_1a^{j}g_2 = g_1^\prime a^{j}g_2^\prime$ becomes 
\begin{equation} \label{eq:long}
(\alpha_1M_1 +h_1)a^j(\alpha_2M_2+h_2) = (\alpha_1'M_1' +h_1')a^j(\alpha_2'M_2'+h_2').
\end{equation}
 
Notice, that $M_1a^jM_2$ and $M_1'a^jM_2'$ are, correspondingly,  the leading monomials in the left-hand  and the right-hand sides of the  equality  above, so $\alpha_1\alpha_2M_1a^jM_2 =\alpha_1'\alpha_2' M_1'a^jM_2'$ which implies, as we mentioned above,  that $M_1 = M_1', M_2 = M_2'$ and $\alpha_1\alpha_2 = \alpha_1'\alpha_2'$.
Therefore, the equation (\ref{eq:long}) can be rewritten as 
$$
\alpha_1M_1a^jh_2  + h_1a^j(\alpha_2M_2+h_2) = \alpha_1'M_1a^jh_2'  +h_1'a^j(\alpha_2'M_2+h_2').
$$
Note, that $M_1$ does not occur as a prefix in any of the monomials in $h_1$, hence it does not occur as a prefix in any of the monomials in $h_1a^j(\alpha_2M_2+h_2)$. Similarly, $M_1$ does not occur in $h_1'a^j(\alpha_2'M_2+h_2')$. It follows that all monomials in $\alpha_1M_1a^jh_2$ as well as in $\alpha_1'M_1a^jh_2'$, and only them, have $M_1$ as a prefix. Hence
 $$
 \alpha_1M_1a^jh_2 = \alpha_1'M_1a^jh_2', \ \ \ h_1a^j(\alpha_2M_2+h_2) = h_1'a^j(\alpha_2'M_2+h_2').
 $$
The first equality implies that $\alpha_1h_2 = \alpha_1'h_2'$. By induction  the second one implies that 
$$
h_1 = \beta h_1', \ \ \ \alpha_2M_2+h_2 = \beta^{-1}(\alpha_2'M_2+h_2'),
$$
so $g_2 = \beta^{-1}g_2'$, therefore $g_1 = \beta g_1'$, as claimed.

  2) holds by the conditions on $f_1,\ldots, f_{k+1}$ and 3) directly follows from 1).

\end{proof}

\begin{lemma} \label{co:main-decomp}
For any $f_1, \ldots, f_{s+1} \in \MA$ there is $m \in \N$ such that $a = a_m$ satisfies the premises of Lemma \ref{le:main-decomp}.
\end{lemma}
\begin{proof}
By direct inspection.
\end{proof}

The following result is an analog of Lemma \ref{le:arbitrary-char} on interpretability of arithmetic in commutative polynomials.

\begin{lemma} \label{le:non-comm-arbitrary-char}
Let $K$ be an arbitrary field and $X$ an arbitrary set with $|X| \geq 2$.  Then the following hold:
\begin{itemize}
\item [1)] For any  non-invertible  element  $P \in \MA_K(X)$ the arithmetic 
$\N = \langle N; +,\cdot, 0,1\rangle$ is interpretable as $\N_P$ (see the beginning of this section) with the parameter $P$  in $\MA_K(X)$ uniformly in  $K$, $X$,  and $P$. 
\item [2)] For any non-invertible  polynomials $P, Q \in \MA_K(X)$ the canonical (unique)  isomorphism of interpretations $\mu_{P,Q}:\N_P \to \N_Q$ is definable in $\MA_K(X)$ uniformly in $K$, $X$,  and $P, Q$.
\item [3)] The arithmetic $\N$ is 0-interpretable in $\MA_K(X)$.
\end{itemize}
\end{lemma}
\begin{proof}
Fix a non-invertible  $P \in \MA_K(X)$. The centralizer $C_\MA(P)$ is defined in $\MA_K(X)$ uniformly in $K$, $X$,  and $P$. By Theorem  \ref{co:Bergman} the ring $C_\MA(P)$  is isomorphic to the ring of polynomials $K[t]$ in one variable $t$. Note that $P$ is still non-invertible in $C_\MA(P)$. By Lemma \ref{le:arbitrary-char-2} one can interpret the arithmetic in $C_\MA(P)$ as $\N_P$ uniformly in $K$ and $P$.  This proves 1).

To prove 2) fix two non-invertible polynomials $P$ and $Q$ in $\MA_K(X)$ and consider  the interpretations $\N_P$ and $\N_Q$ from 1). One needs to show that the canonical isomorphism $\mu_{P,Q}: \N_P \to \N_Q$ which is defined by the map $P^n  \to Q^n$ is uniformly definable in $K,X,P$ and $Q$.

Fix some particular $e\geq 3$. By Corollary \ref{co:main-decomp}  for any natural positive $s$ there exists $m  = m(s,P,Q) \in \N$ such that $a = a_m$ satisfies the premises of Lemma \ref{le:main-decomp}  when $f_0 = 1, f_i = P^iQ^i, i = 1, \ldots,s, f_{s+1} = 1$. 

Consider an element
\begin{equation}\label{eq:main-decomp-PQ}
f =f_s =  a^ePQa^{e+1}P^2Q^2a^{e+2}P^3Q^3 \ldots  P^sQ^sa^{e+s}.
\end{equation}
This $f$ satisfies the following conditions:
\begin{itemize}
\item [(1)] $f = a^ePQa^{e+1}g_3$, where $g_3 \neq ag_3^\prime$ for any $g_3^\prime \in \MA$. 
\item [(2)] if $f = g_1a^{e+i}g_2a^{e+i+1}g_3$ where $g_1\neq g_1^\prime a$ for any $g_1^\prime \in \MA$,  $ g_2 \neq ag_2^\prime, g_2  \neq g_2^{\prime \prime} a$ for any $g_2^\prime, g_2^{\prime \prime} \in \MA$,  and $g_3$ as above,  then either $g_3 = Pg_2Qa^{e+i+2}g_4$ for some $g_4 \neq ag_4^\prime$ or $g_3 \in K$.
\item [(3)] for every $i \in \N$ if $f = g_1a^{e+i}g_3$ for some $g_1, g_3$ as above   then such $g_1, g_3$ are uniquely defined (up to a multiplicative constant from $K$). 
\item [(4)] $f = g_1a^{e+s-1}ua^{e+s}$ for some $g_1$ as above and $u$ such that $ u \neq au^\prime, u  \neq u^{\prime \prime} a$ for any $u^\prime, u^{\prime \prime} \in \MA$. In this case $u = \gamma P^sQ^s$ for some $\gamma \in K$. 
\end{itemize}
Indeed, (1) and (4) hold by construction, (3) follows from Lemma \ref{le:main-decomp} (item 1). We claim that (2) also comes from  Lemma \ref{le:main-decomp}, since in this case $a^{e+i}$ and $a^{e+i+1}$ are maximal $a$-occurrences in $f$, so they  are uniquely defined in $f$. Hence by  Lemma \ref{le:main-decomp} (item 1) one has 
$$
g_1a^{e+i}g_2 = \alpha a^ePQa^{e+1}P^2Q^2a^{e+2}P^3Q^3 \ldots  P^iQ^i,  
$$ 
and 
$$
 g_3 = \beta  P^{e+i+2}Q^{e+i+2} \ldots  P^sQ^sa^{e+s}
 $$
for some $\alpha, \beta \in K$, which proves the claim.

Conditions (1) - (4) can be written by a formula $\psi(f,a,m,u,P,Q)$ since  the set of all pairs $\{(\alpha a_m,m) \mid  m \in \N_{x_2}, \alpha \in K\}$ is definable in $\MA$ by some formula $A(a,m)$ in $L_X$ (Lemma \ref{le:a-m}) and the operation $(a_m,i) \to a^i$ is also definable by 1)  from this theorem. Then the formula 
$$\psi_1(f,a,u,P,Q) = \exists m(\psi(f,a,m,u,P,Q) \wedge A(a,m))$$
 defines  elements $f$, $a$, $u$  for which there exists a decomposition (\ref{eq:main-decomp-PQ}) satisfying the conditions (1)-(4), in particular, $u = \gamma P^sQ^s$ for some $s$ and $\gamma \in K$.  

Recall that the sets $\N_P = \{P^m \mid m \in \N \}$ and $\N_Q = \{Q^m \mid m \in \N \}$  are definable in $\MA_K(X)$ with the parameters $P,Q$, as we noticed above.  Now the formula
$$
\psi_0(n_1,n_2,P,Q) = \exists f \exists a \exists u \exists \gamma (\psi_1(f,a,u,P,Q) \wedge u = \gamma n_1n_2 \wedge \gamma \in K \wedge n_1 \in \N_P \wedge n_2 \in \N_Q)
$$
defines $\mu_{P,Q}:\N_P \to \N_Q$.

3) follows from 2) as in Lemma \ref{le:arbitrary-char}.
\end{proof}

 \subsection{Interpretation of $S(K,\N)$ in $\MA_K(X)$}
In this section  $K$ is  an infinite field and $X$ is a set with $|X| \geq 2$.  Our goal is to prove an analogue of Theorem \ref{Bauva} in the non-commutative case. 

As was discussed in Section \ref{se:N-in-A} for any non-invertible polynomial  $P \in \MA$ the one-variable polynomial  ring   $K[P]$ is definable with the parameter $P$ in  $\MA_K(X)$  uniformly in $K$, $X$,  and $P$.   By Theorem  \ref{Bauva}  the model $S(K,\N)$ is interpretable in the ring $K[P]$  uniformly in $K$ and $P$, hence  $S(K,\N) =\langle K, S(K),\MN; t(s,i,a), l(s), \frown\rangle$ is interpretable in $\MA$ with a parameter $P$ uniformly in $K,X$ and $P$. Denote this interpretation by 
 $$
 S(K,\N)_P = \langle K, S(K)_P,\MN_P; t_P(s,i,a), l_P(s), \frown_P\rangle.
 $$
  Recall, that the set  $S(K)_P$ of all finite sequences $s = (\alpha_0, \ldots, \alpha_n)$ in $K$ is interpretable in $K[P]$  as the set of all pairs of the the type $(\Sigma_{i=0}^n \alpha_iP^i,P^n)$, where $\alpha_i \in K, n \in \N$. 
The predicate $t_P(s,i,a)$ and the operations $l_P(s), \frown_P$ are defined  in $\MA$ by some formulas (see Theorem \ref{Bauva}) which we denote by $\phi_t(s,i,a,P), \phi_\ell(s,P)$, and $\phi_\frown(s_1,s_2,P)$, correspondingly.  
Our goal is to show that all the interpretations $S(K,\N)_P$ are definably isomorphic in $\MA$. By an isomorphism $\nu: S(K,\N)_P \to S(K,\N)_Q$ we understand  a pair  of isomorphisms:  the identity isomorphism $K \to K$ and the unique isomorphism $\N_P \to \N_Q$  (the unique isomorphism of interpretations of the arithmetic $\N$ which sends $ 1$ in $\N_P$ to $1$ in $\N_Q$). We refer to this pair of isomorphisms as to the {\em canonical} isomorphism of the interpretations.

 \begin{theorem} \label{th:S(F,N)-non-comm} 
 Let  $K$ be an infinite field and $X$ an arbitrary  set with $|X| \geq 2$. Then the following hold:
 \begin{itemize}
\item [1)]  for a given non-invertible polynomial $P \in \MA_K(X) $ one can interpret $S(K,\N)$ in $\MA_K(X) $  by $S(K,\N)_P$ as  above, using the parameter $P$ uniformly in $K$, $X$, and $P$. 
\item [2)] for any non-invertible polynomials $P, Q \in \MA_K(X) $ the canonical  isomorphism of interpretations $\nu_{P,Q} : S(K,\N)_P \to S(K,\N)_Q$ is definable in $\MA_K(X) $ uniformly in $K$, $X$, $P$, and $Q$.
\item [3)]  $S(K,\N)$ is 0-interpretable  in $\MA_K(X) $  uniformly in $K$ and $X$.
 
 \end{itemize}
 \end{theorem}
\begin{proof}

1) was shown already at the beginning of this section. 

To prove 2) observe that by Lemma \ref{le:non-comm-arbitrary-char} for any such $P$ and $Q$ there is a formula that defines the set of pairs $R = \{(P^m,Q^m) \mid m \in \N\}$ uniformly in $K$, $X$,  $P$, and $Q$.
Recall that a sequence $s = (\alpha_0, \ldots,  \alpha_m)$ is interpreted in $S(K,\N)_P$ as a pair $s_P = (\sum_{i = 0}^m \alpha_iP^i, P^m) \in S(K)_P$,   and similarly, by the pair $s_Q = (\sum_{i = 0}^m \alpha_iQ^i, Q^m) \in S(K)_Q$ in $S(K,\N)_Q$.  We need to show that the set of pairs $\{(s_P,s_Q) \mid s \in S(K)\}$ is definable in $\MA$ uniformly in $K, X, P, Q$. Since the set of pairs $R$ is definable it follows that the set of pairs  
$(s_P,r_Q)$ such that $s,r \in S(K)$ and $l_P(s_P) = l_Q(r_Q)$ (i.e., the lengths of the tuples $s$ and $r$ are equal)  is  definable in $\MA$ uniformly in $K, X, P, Q$.  Recall that the predicate $t_P(s,i,a)$ defines in $S(K,\N)_P$ the coordinate functions $ s_P  \to a  \in K$, where $a$ is the $i$'s term of  the sequence $s_P$,  uniformly in $K, i, X,P$ (here $0 \leq i \leq l(s)$ and $K$ is viewed as the set of invertible elements in $\MA$). Therefore, there is a formula which states that for any $0 \leq i \leq l(s) = l(r)$ the sequences   $s_P$ and $r_Q$ have the same $i$ terms. Hence   
the set of pairs 
$$
\{ (\sum_{i = 0}^m \alpha_iP^i, \sum_{i = 0}^m \alpha_iQ^i) \mid \alpha_i \in K, m \in \N\}
$$
 is also definable in $\MA_K(X)$ uniformly in $K, X, P$ and $Q$. This gives an isomorphism $ S(K,\N)_P \to S(K,\N)_Q$ definable in $\MA_K(X) $ uniformly in $K$, $X$, $P$, and $Q$, as claimed.

 This completes the proof that the canonical  isomorphism of interpretations $\nu_{P,Q} : S(K,\N)_P \to S(K,\N)_Q$ is  definable in $\MA_K(X) $ uniformly in $K$, $X$, $P$, and $Q$.

3) follows from 2).

\end{proof}

\subsection{Definable isomorphisms of centralizers}

In this section  $K$ is  an infinite field and $X$ is a set with $|X| \geq 2$. 
We say that a non-invertible polynomial $P \in \MA_K(X)$ {\em self-generates} its own centralizer $C_\MA(P)$ if $C_\MA(P) = K[P]$.

\begin{theorem} \label{th:def-centr-iso}
Let  $K$ be an infinite  field and $X$ an arbitrary  set with $|X| \geq 2$. Then the following hold:
\begin{itemize}
\item [1)] The subset of non-invertible polynomials that self-generate their own centralizers in $\MA_K(X)$ is 0-definable in $\MA_K(X)$ uniformly in $K$ and $X$.
\item [2)]  for any non-invertible   polynomials $P, Q \in \MA_K(X)$ that self-generate their own centralizers  there exists a formula $Is(x,y,P,Q)$ which defines the isomorphism $\Sigma_{i = 0}^n \alpha_i P^i \to  \Sigma_{i = 0}^n \alpha_i Q^i $  of the centralizers $ C_\MA(P)$ and $C_\MA(Q)$ uniformly in $K,X,P$ and $Q$.

\item [3)] The one-variable polynomial ring $K[t]$ over $K$ is 0-interpretable in $\MA_K(X)$ via all proper   centralizers   in $\MA_K(X)$. 
\end{itemize}
\end{theorem}
\begin{proof}

We claim that the set of  non-invertible polynomial $P  \in \MA_K(X)$ that self-generate their own centralizers  is 0-definable in $\MA_K(X)$.  Indeed, the polynomial subring  $K[P]$ is definable in the polynomial ring $C_\MA(P)$ by Lemma \ref{le:int-poly-one}. So one can write a formula $\Delta(P)$ that  states that $C_\MA(P) = K[P]$ uniformly in $K,X$ and $P$. This proves 1).

To see  2) let $P$ and $Q$ be non-invertible polynomials in $\MA$ that self-generate their own centralizers.  In the proof of Theorem \ref{th:S(F,N)-non-comm}  we showed that the set of  pairs 
$$
\{ (\sum_{i = 0}^n \alpha_iP^i, \sum_{i = 0}^n \alpha_iQ^i) \mid \alpha_i \in K, n \in \N\}
$$
is definable in $\MA_K(X)$ uniformly in $K, X, P$ and $Q$. But this is precisely the graph of an  isomorphism $C_\MA(P) \to C_\MA(Q)$, as claimed. 
 
 3) follows from 2) as was mentioned above (see the corresponding argument in the proof of Lemma \ref{le:arbitrary-char}). Indeed, it suffices to notice that every proper centralizer $C_{\MA_K(X)}(y)$ in $\MA_K(X)$ is the centralizer of some non-invertible polynomial $P$ in $\MA_K(X)$ that self-generate this centralizer. This condition on $y$ and $P$ can be written by a formula uniformly in $K,X,y,P$.
\end{proof}

Now we are ready to prove the following  result which is important  for our study of model theory of free associative algebras.

\begin{theorem} \label{th:centralziers}
 There exists a sentence $Isom$ of the language of ring theory $L$  such that:
\begin{itemize}
\item [1)] $\MA_K(X) \models Isom$ for any  infinite  field $K$ and any   set $X$ with $|X| \geq 2$.
\item [2)] for any unitary ring $A$ if $A \models Isom$ then all proper centralizers of $A$ of the type $C_A(P)$ where $P \in A$  are isomorphic.
\end{itemize}
\end{theorem}
\begin{proof}
Let $\Delta(x)$ be a formula from Theorem \ref{th:def-centr-iso} item 1) which defines in $\MA_K(X)$ the set $\mathcal C$ of all non-invertible polynomials $P  \in \MA_K(X)$ that self-generate their own centralizers. Consider the  following conditions:
\begin{itemize}
\item for any element $x$ such that $C_A(x) \neq A$ there exists $P$ such that $\Delta(P)$ holds and $C_A(P) = C_A(x)$.
\item for any elements $P, Q$ which both satisfy the formula $\Delta(x)$ the formula $Is(x,y,P,Q)$ defines a map $x \to y$ which is an isomorphism of the centralizers $C_A(P)$ and $C_A(Q)$.
\end{itemize}
Note that the conditions above can be written by a sentence $Isom$ in the ring language $L$ in a such a way that $Isom$ satisfies the conditions 1)-2) from the conclusion of the theorem. Indeed, the formulas $\Delta(x)$ and $Is(x,y,P,Q)$ are given in Theorem \ref{th:def-centr-iso} and they hold in such $\MA_K(X)$. The centralizers of the type  $C_A(x)$  can be described by formulas with parameters $x$ (the ring $A$ is not involved, of course). To write all the other conditions is a straightforward exercise.  This proves the theorem.

\end{proof}


\subsection{Definability of bases in $\MA_K(X)$}

We continue to use notation from the previous sections. In particular, below   $K$ is  an infinite field,  $X = \{x_1, \ldots,x_n\}$ is a finite set with $n = |X|$, $\MA = \MA_K(X)$.

In  Lemma \ref{le:list-superstructure} we described how one can   0-interpret  the superstructure $S(\N,\N)  = \langle \N, S(\N),\MN; t(s,i,a), l(s), \frown, \in \rangle$  in $\N$. Fix a particular such interpretation and denote it by  $S(\N,\N)^\ast$. This allows us to assume that  the tuples from $S(\N)$ and operations and predicates from $S(\N,\N)$ are 0-interpretable  in $\N$. Furthermore, as was mentioned right after  Lemma \ref{le:list-superstructure} in the  interpretation $S(\N,\N)^\ast$ the set of tuples $S(\N)$ is  interpreted by a 0-definable subset of $\N$ (by the set of the codes of these tuples with respect to some fixed efficient enumeration of the tuples). 
 
 Consider the following interpretation of the free monoid $\MM_X$  in $S(\N,\N)$.  A monomial $M=x_{i_1}\ldots x_{i_m}\in\MM_X$ can be uniquely  represented by a tuple of natural numbers $ t_M=(i_1,\ldots, i_m)$.   Denote by $T$ the set of all tuples $t= (t_1, \ldots,t_m) \in S(\N)$  such that for any $i$ one has $1\leq t_i \leq n$.  Conversely,  with  any tuple $t= (t_1, \ldots,t_m) \in T$  one can associate a  monomial $M_t = x_{t_1} \ldots x_{t_m} \in \MM_X$. The multiplication in $\MM_X$  corresponds to  concatenation of  tuples  in  $T$, which is 0-definable in $S(\N,\N)$. The construction above gives a  0-interpretation of $\MM_X$  in $S(\N,\N)$. Combining this interpretation with the interpretation  $S(\N,\N)^\ast$ of  $S(\N,\N)$ in $\N$ one gets a 0-interpretation of $\MM_X$ in $\N$ which we denote by $\MM_X^*$. Observe that the map $M \to t_M$ gives rise to an isomorphism $\MM_X \to \MM_X^*$,   termed standard.
 
 By Lemma \ref{le:non-comm-arbitrary-char} for any  non-invertible  element  $P \in \MA_K(X)$ the arithmetic 
$\N$ is interpretable in $\MA$  as the structure $\N_P$  with the parameter $P$  uniformly in  $K$, $X$,  and $P$.  Combing this interpretation with the interpretation  $S(\N,\N)^\ast$ of  $S(\N,\N)$ in $\N$ one gets an interpretation of $S(\N,\N)$  in $\MA_K(X)$ with the parameter $P$  uniformly in  $K$, $X$,  and $P$, we denote this interpretation by ${S(\N,\N)_P}^\ast$.
Similarly,   the interpretation $\MM_X^*$ of $\MM_X$ in   $S(\N,\N)$ and the interpretation ${S(\N,\N)_P}^\ast$ gives rise to an interpretation of the monoid $\MM_X$ in $\MA_K(X)$ with the parameter $P$  uniformly in  $K$, $X$,  and $P$. We denote this interpretation by $\MM_{X,P}^*$. 
Composition of the standard isomorphism $\MM_X \to \MM_X^*$ above and the (unique) isomorphism of $\N$ and $\N_P$ one gets an isomorphism $\MM_X \to \MM_{X,P}^*$ which we again call the standard one. The set of tuples $T$ in $S(\N,\N)$ on which we based the interpretation $\MM_X^*$ is mapped by the standard isomorphism onto some subset of $\N_P$ which we denote by $T_P$.

 Recall that in Lemma \ref{le:monomials} we showed that the submonoid $K\MM_X = \{\alpha M \mid \alpha \in K, M \in X^\ast\} \leq \MA$ is definable  with parameters $X$ in $\MA$, while the free monoid $\MM_X$ is interpretable in $\MA$ as $K\MM_X/K$ with parameters $X$ uniformly in $K$ and $|X|$. 
 
 Clearly the structure  $\MM_{X,P}^*$ is very different from $K\MM_X/K$, though isomorphic. The next result shows that they are definably isomorphic inside $\MA$ in the language $L_X$.
 

\begin{lemma} \label{le:M-X-c}
Let $c \in X$. In the notation above the following hold:
\begin{itemize}
\item [1)] There is a formula $\Phi(y,z,X,c)$ of the language $L$ with parameters $X$ and $c$ such that 
for any elements $t,u \in \MA$ the formula $\Phi(t,u,X,c)$ holds in $\MA$ if and only if $t \in T_c$ and $u = \alpha M_t$ for some $\alpha \in K$.   
\item [2)] The standard  isomorphism $ \MM_{X,c}^* \to K\MM_X /K $ defined by the map $t \to KM_t/K$ is definable in $\MA_K(X)$ with parameters $X$ uniformly in $K$, $|X|$, and $c$.
\end{itemize}
\end{lemma}
\begin{proof}
The case $|X| = 1$ was done in Lemma \ref{le:arbitrary-char}. Assume now that $|X| \geq 2$. Without loss of generality we may assume that $c = x_2$.
Below we construct a formula $\Phi(y,z,c)$ of the language $L_X$  such that 
$\MA_K(X) \models \Phi(t,u,c)$ if and only if  $t \in T_c$ and  $u \in KM_t$. 

Notice that the set $T_c$  is 0-definable in $\N_c$ uniformly in $K, |X|$, and $c$.   Notice that the length function $\ell: T_c \to \N_c$ that gives the length of a tuple $t \in T_c$ is 0-definable in $S(\N,\N)_c$, as well as in $\MA_K(X)$ (this time with the parameter  $c$). 
  Hence there is a formula $\phi_1(y,y_1,c)$ in $L$ such that in the notation above $\MA_K(X) \models \phi_1(t,m,c)$ if and only if $t \in T_c$ and $m = \ell(t)$. Similarly, there exists a formula $\phi_2(y,y_2,y_3, c)$ in the language $L$ such that $\MA_K(X) \models \phi_2(t, i,s,c)$ if and only if $t \in T_c, i,s \in \mathbb{N}_c$, $1 \leq i \leq \ell(t)$, and $s$ is the $i$'s component of the tuple $t$.

Now for a tuple $t = (t_1, \ldots,t_m)\in T_c$ and a fixed number $p \geq 3$ define a word $w_t$ as follows, where $a = a_m$ defined in (\ref{eq:a}).  
\begin{equation}\label{eq:w}
w_t =a^px_{t_1}a^{p+1}x_{t_1}x_{t_2}a^{p+2}\ldots  a^{p+m-1}x_{t_1}\ldots x_{t_m}a^{p+m}.
\end{equation}

The monomial  $w_t$ is completely determined by the  tuple $t$ and the  following conditions:
\begin{itemize}
\item [a)] (head) $w_t = a^px_{t_1}a^{p+1}v$ for some $v \in \MM_X \smallsetminus a\MM_X$ ($v$  does not have $a$ as its prefix);
\item [b)] (tail) $w_t = w_1a^{p+m}$ where $m = \ell(t)$ and $w_1 \in \MM_X \smallsetminus \MM_Xa$ ($w_1$ does not have $a$ as its suffix);
\item [c)] (recursion) for any $i \in \N_c, 0 <  i < m$,  and any $ w_1,w_2,w_3 \in \MM_X$ such that $w_1$ does not have $a$ as its suffix, $w_2$ does not have $a$ neither as its suffix or prefix, and $w_3$ does not have $a$ as its prefix,  if $w=w_1a^{p+i -1}w_2a^{p+i}w_3$ then $w_3 = w_2x_{t_i}a^{p+i+1}v_1$ for some $v_1 \in \MM_X$ which does not have $a$ as its prefix. 
\item [d)] (uniqueness) if $w_t = w_1a^jw_2 = w_1^\prime a^jw_2^\prime$ for some $w_1,w_1^\prime \in \MM_X \smallsetminus \MM_Xa$, $w_2,w_2^\prime \in \MM_X \smallsetminus a\MM_X$ and  $j <m$ then $w_1 = w_1^\prime, 
 w_2 = w_2^\prime$.
\end{itemize}
Indeed, the product (\ref{eq:w}) satisfies the conditions of Lemma \ref{le:main-decomp}, since $a = a_m$ and $m = \ell(t)$  which is the degree of the polynomial $M_t$. Now the required uniqueness follows from  Lemma \ref{le:main-decomp}.  

Now if in the conditions a)-d) we replace $\MM_X$ by $K\MM_X$, and $a$ by any element from $Ka$ then the new conditions, say a') - d'), define not only the element $w_t$ but all the elements of the type $\alpha w_t$, where $\alpha \in K$ and only them. 
As in Lemmas \ref{le:a-m} and \ref{le:non-comm-arbitrary-char}  one can write down the condition a') -d') by formulas of the language $L_X$. Indeed, the only extra required tools  that did not occur in the  arguments in Lemmas \ref{le:a-m} and \ref{le:non-comm-arbitrary-char} are the ones that allow one to write down that  $m= \ell(t)$, to describe the components $t_i$ of $t$, and to write down the conditions on $w_i$ such as $w_i \in KM_X$,$w_i \not  \in aKM_X$,  or $w_i \not \in KM_Xa$. This can be done with the use of the formulas $\phi_1$ and $\phi_2$, and formula $\phi(a,X)$  from the proof of   Lemma \ref{le:monomials}. It follows that there is a formula $\phi_3(y,y_1, y_4, y_5,c)$ in the language $L_X$ such that $\phi_3(t,m,b,w,c)$ holds in $\MA$ on elements $t,m,a,w,c$ if and only if $t \in T_c, m \in \N_c, m = \ell(t), b = \alpha a_m, w = \beta w_t$ for some $\alpha, \beta \in K$. 

Notice, that by construction $w_t = w_1a^{p+m -1}M_ta^{p+m}$ for some $w_1$  such that $w_1 \in \MM_X \smallsetminus \MM_Xa$, and such $w_1$ is unique by the condition  d). So if  elements  $t,m,b, w \in \MA$  satisfy in $\MA$ the formula $\phi_3(t,m,b,w,c)$  then the condition that some element $u \in \MA$ is equal to  $ \gamma M_t$ for some $\gamma \in K$  is equivalent to the condition that $\exists w_1 (w= w_1a^{p+m -1}ua^{p+m})$, which can be described by a formula, say $\phi_4(y,y_1, y_4, y_5,z,c)$, in the language $L_X$. it follows that the formula 
$$
\phi_5(y,y_1, y_4, y_5,z,c)   = \phi_3(y,y_1, y_4, y_5,c) \wedge \phi_4(y,y_1, y_4, y_5,z,c)
$$
holds in $\MA$ on elements $t,m,b, w, u$ if and only if $t \in T_c, m \in \N_c, m = \ell(t), b = \alpha a_m, w = \beta w_t, u = \gamma M_t$ for some $\alpha, \beta, \gamma  \in K$.

 Hence the  formula 
 
 $$
 \Phi(y,z,c) = \exists y_1 \exists y_4 \exists y_5 \phi_5(y,y_1, y_4, y_5,z,c) 
 $$
defines all the pairs $(t,\gamma M_t)$ for $t \in T_c$, and $\gamma \in K$, as required in 1).

This formula defines an isomorphism of interpretations $ \MM_{X,c} \to \MM_X$ given  by the map $t \to Kt_M/K$, which proves 2).

\end{proof}

Building on the interpretation  $\MM_X^*$ of $\MM_X$ in $S(\N,\N)$  we interpret $\MA_K(X)$ in $S(K, \N)$ as follows. Let $\MM_{X,c}$ be the interpretation of $\MM_X$ in $S(\N,\N)^\ast$ as above.   For an element  $f=\sum_{i= 1}^s \alpha _iM_i \in\MA_K(X)$, where $\alpha_i \in K, M_i \in \MM_X$,  we associate a pair  $q_f = (\overline\alpha, \overline t)$, where $\overline\alpha = (\alpha_1, \ldots, \alpha_s)$, $\overline t = (t_{M_1}, \ldots,t_{M_s})$.  This gives interpretation, say $\MA_K(X)^*$, of $\MA_K(X)$ in $S(K,\N)$, hence by transitivity of interpretations, interpretation  $\MA_K(X)^{**}$ in $\MA_K(X)$.

To proceed we need a notation. The set $S(T)$ of all tuples of elements (which are also tuples) of $T$. The interpretation of $S(\N,\N)$ in $\N$ above, and of $\N$ in $\N_P$ gives the corresponding image $S(T)_P$ of $T$ in $\N_P$. For a tuple $s = (t_1, \ldots,t_e) \in S(T)$, as well as for $s \in S(T)_P$,  we introduce the following set of polynomials in  $\MA$:
$$
B(s,K) =  KM_{t_1} + \ldots + KM_{t_m}  = \{ \alpha_1M_{t_1} + \ldots \alpha_m M_{t_e} \mid \alpha_i \in K \}.
$$
Notice, that whether $s \in S(T)$ or $s \in S(T)_P$ the set $B(s,K)$ is the same. Also the set $B(s,K)$  depends only on the set $\hat s$ formed by all the coordinates of the tuple $s$, i.e., if $s_1, s_2 \in S(T)$ are such that $\hat{s_1} = \hat{s_2}$ then $B(s_1,K) = B(s_2,K)$.

\begin{lemma} \label{le:s-and-f} 
Let $\MA = \MA_K(X)$ and $c \in X$.  There exists a formula $\Psi(y,z,c,X)$ in  the ring language $L$  with parameters $c$ and $X$ such that for any $s, f \in \MA$ 
$$
\MA \models \Psi(s,f,c,X) \Longleftrightarrow s \in S(T)_c  \ and \  f  \in B(s,K)
$$
\end{lemma}
\begin{proof}

Fix a  tuple  $s= (t_1, \ldots,t_e) \in S(T)_c$. 
Let $m= \max\{\ell(t_i) \mid i = 1, \ldots,s\}$ and $a= a_m \in \MM_X$ defined in (\ref{eq:a}). As in Lemma \ref{le:M-X-c} fix a number  $p \geq 3$.

For a polynomial $f = \alpha_1M_{t_1} + \ldots + \alpha_eM_{t_e} \in B(s,K)$ consider the following polynomial 
\begin{equation} \label{eq:p-unique}
\hat f=a^{p+1}h_1a^{p+1}h_2\ldots a^{p+e}h_{s}a^{p+e+1},
\end{equation}
where $h_1=\alpha _1M_1$ and $ h_{i+1}=h_i+\alpha _{i+1}M_{i+1}$ for $ i = 1, \ldots, e-1$. Observe, that by construction  $h_e = f$.

If $e = 1$ then the polynomial $\hat f$ is completely  determined up to a multiplicative factor $\alpha \in K$ by the conditions $\hat f = a^{p+1}h_1a^{p+1}$ and $h_1 \in KM_{t_1}$.

Let $e\geq 2$. The polynomial $\hat f$ is completely determined up to a multiplicative factor $\alpha \in K$ by the following conditions:
\begin{itemize}
\item [i)] $\hat f = a^{p+1}h_1a^2h_2a^3g$, where $g \neq ag_1$ for any $g_1 \in \MA_K(X)$;
\item [ii)] for any $1\leq i \leq e$ if  $\hat f = g_1a^ig_2 = g_1^\prime a^ig_2^\prime$ then $g_1 = g_1^\prime, g_2 = g_2^\prime$ (up to a multiplicative constant from $K$).  (Follows from Lemma \ref{le:main-decomp}.)
\item [iii)] If $\hat f = g_0a^ig_1a^{i+1}g_2a^{i+2}g_3$ and $g_2 \neq ag_4$ for any $g_4  \in \MA_K(X)$,  $g_2-g_1=\alpha _{i+1}M_{i+1}$, then $g_3 = h_{i+2}a^{i+2}g_4$, where $g_4 \neq ag_5$ for any $g_5 \in \MA_K(X)$ , and $h_{i+2} = g_2 +\alpha_{i+2}M_{i+2}$.
\item [iv)] $\hat f = g_5a^{e+1}$, where $g_5 \neq g_6a$ for any $g_6  \in \MA_K(X)$.
\end{itemize}

Indeed, to show that i)-iv) determine $\hat f$ completely up to a multiplicative factor $\alpha \in K$ one needs the uniqueness of the decomposition (\ref{eq:p-unique}), which follows from Lemma \ref{le:main-decomp}.

Now it follows from the argument above that a polynomial $g \in \MA$ has the form $\hat f$ for some $f \in B(s,K)$ if and only if it satisfies the following conditions.

\begin{equation} \label{eq:p-unique-2}
g=a^{p+1}h_1a^{p+1}h_2\ldots a^{p+e}h_{e}a^{p+e+1},
\end{equation}
 where $h_1 \in KM_{t_1}, h_{i+1}=h_i+h_i'$ and $h_i' \in KM_{t_{i+1}}$, and $g$ and $h_i, h_i'$ satisfy the new  conditions  i') - iv') obtained from the conditions i)-iv) by replacing in  iii) the condition $h_{i+2} = g_2 +\alpha_{i+2}M_{i+2}$  with the new one: $h_{i+2} = g_2 +g_2'$ where $g_2' \in KM_{i+2}$, and leaving everything else the same.  
 
 By Lemma \ref{le:M-X-c} there is a formula  $\Phi(y,z,X,c)$ of the language $L$ with parameters $X$ and $c$ such that  for any elements $t,u \in \MA$ the formula $\Phi(t,u,X,c)$ holds in $\MA$ if and only if $t \in T_c$ and $u = \alpha M_t$ for some $\alpha \in K$. This allows one to describe the new conditions  i') - iv')  by a formula 
in the language $L$ with parameters $X$ and $c$. 
Therefore there is a formula $\Psi_1(y,z_1,c,X)$ in the language $L$ with parameters $X$ and $c$ such that for any $s, g \in \MA$ 
$$
\MA \models \Psi(s,g,c,X) \Longleftrightarrow s \in S(T)_c  \ and \  g  = \hat f \ for \ some \ f  \in B(s,K).
$$
Observe also that given an element $ g = \hat f$ for some $f \in B(s,K)$, the element $f$ is completely determined up to a multiplicative factor from $K$ by the conditions that $g =  g'a^{p+e}h_{e}a^{p+e+1}$, $e = \ell(s)$, $g' \not \in \MA a$, $h_e \not \in a\MA$, $h_e \not \in \MA a$, and $h_e = f$. We denote these conditions by  v). All these conditions v) again can be described by a formula, say $\Psi_2(y,z_1,z,c,X)$ in  the language $L$ with parameters $X$ and $c$, so that for any $s, g,f \in \MA$ the formula $\Psi_2(s,g,f,c,X)$ holds in $\MA$ if and only if $s, g,f$ satisfy the condition v).

Clearly, the following formula in  the language $L$ with parameters $X$ and $c$
$$
\Psi(y,z, c,X) = \exists z_1 (\Psi_1(y,z_1,c,X) \wedge \Psi_2(y,z_1,z,c,X)
$$
holds in $\MA$   on elements $s, f \in \MA$ if and only if 
$s \in S(T)_c$ and $ f  \in B(s,K)$. This proves the lemma.

\end{proof}

Now we are ready to prove the main result of this section. 

 \begin{theorem} \label{th:bases} 
 Let $K$ be an infinite field and $X$ an arbitrary finite set. Then the set of all free bases of $\MA_K(X)$ is 0-definable in $\MA_K(X)$.
 \end{theorem}
 
 \begin{proof} 
 If $|X| = 1$ then by Lemma \ref{le:int-poly-one} there is a formula $ \psi(y,z)$ of the ring language $L$  such that   for any non-invertible polynomial $P \in K[X]$ the formula  $ \psi(y,P)$ defines in $K[X]$ the polynomial ring $K[P]$. Hence $P$ generates $K[X]$, i.e., forms a basis for $K[X]$, if and only if $K[P] = K[X]$, which is equivalent to the condition that $P$ satisfies the following formula
 $$
 Gen(P) = \forall a \exists y (\psi(y,P) \wedge a = y).
 $$ 
 This proves the theorem in the case $|X| = 1$.
 
 Suppose now that $|X| = n \geq 2$. Observe that a set of  elements $V = \{v_1, \ldots, v_n\} \subset  \MA_K(X) $  forms a basis in  $\MA_K(X)$  if and only if the set of all monomials $\MM_V$ in $V$  is a basis of the $K$-vector space $\MA_K(X)$, i.e., 
 
 \begin{itemize}
 
 \item [1)]  every element of $\MA$ is a $K$-linear combination of elements from $\MM_V$, and 
 
 \item [2)] elements from $\MM_V$ are $K$-linearly independent. 
 
 \end{itemize}
 
  In the notation of Lemma \ref{le:s-and-f} one can write the conditions above as follows.  
 
 \begin{itemize}
 \item [1')]  $\forall a\in \MA \exists s \in S(T)_{v_2} (a \in B(s,K))$,
 \item [2')] $\forall s_1, s_2 \in S(T)_{v_2} (\hat{s_1} \neq \hat{s_2} \rightarrow B(s_1,K) \cap B(s_2,K) = \{0\})$.
 \end{itemize}
 Here $S(T)_{v_2}$ is the set of tuples $s= (t_1, \ldots,t_e)$ from the set $S(T)_{v_2}$ which is the interpretation of the set $S(T)$ in $\N_{v_2}$ (in Lemma \ref{le:s-and-f}  it was the interpretation $\N_c$ with $c = x_2$).  Observe, that the set $S(T)_{v_2}$ is definable in $\MA_K(X)$ with parameters $V$ uniformly in $K, X,$ and $V$.  By Lemma  \ref{le:s-and-f}  the formula $\Psi(s,f,c,V)$ (from this lemma) with parameters $V$  allows one to write down the conditions $a \in B(s,K)$ and $B(s_1,K) \cap B(s_2,K) = \{0\}$ from the above.  The set $\hat{s}$ is definable uniformly with parameter $s$ in $\N$ and in $\N_{v_2}$, so the condition  $\hat{s_1} \neq \hat{s_2}$ is also definable.  This gives a formula $Gen(V)$ which defines the set of bases in $\MA_K(X)$ uniformly in $K$ and $|X|$.  This proves the theorem.
 
 \end{proof}

\section{Tarski-type questions for free associative algebras}
 
 
 In this section we assume that all free associative algebras have non-zero rank.

\begin{theorem} \label{th:undecidable}
The first-order theory of $\MA_K(X)$ is undecidable for any filed $K$ and a non-empty set $X$.
\end{theorem}
 \begin{proof}
 The result follows from  Theorem \ref{Rob} and Theorem \ref{th:def-centr-iso} item 3).
 \end{proof}

Next we address the question about canonical elementary embeddings of free associative algebras of different ranks over the same filed. Recall that a substructure   $A$ of a structure   $B$ of a language $\L$ is an {\it elementary substructure } if for any formula $\phi(x_1, \ldots,x_n)$ of the  language $\L$ and for any elements $a_1, \ldots, a_n \in A$ the formula $\phi(a_1, \ldots,a_n)$ holds in $B$ if and only if it holds in $A$. By   Tarski-Vaught test a substructure $A$ is an elementary substructure of $B$ if for any formula $\phi(x,a_1, \ldots,a_n)$ with parameters $a_1, \ldots,a_n \in A$ the formula $\exists x \phi(x,a_1, \ldots,a_n)$ holds in $B$ if and only if there exists $a \in A$ such that  $\phi(a,a_1, \ldots,a_n)$ holds in $B$. 

The following result, in the case when the set $X$ is infinite, is known in the folklore (in the general case of free algebras in a variety). Nevertheless, we give a proof for both cases for the sake of completeness.

\begin{theorem} \label{th:elem-submodel}
Let $X$ and $Y$ be disjoint non-empty sets  and $K$ an arbitrary field. Then $\MA_K(X)$ is an elementary subring of $\MA_K(X \cup Y)$ if and only if $|X| = \infty$.  
\end{theorem}
\begin{proof}
If the set $X$ is finite then Theorem \ref{th:rank-non-com} shows that  $\MA_K(X \cup Y)$

To prove that $\MA_(X)$ is an elementary subring of $\MA_K(X \cup Y)$  we use the Tarski-Vaught test. Let  $\phi(x,a_1, \ldots,a_n)$ be a ring language formula with parameters $a_1, \ldots,a_n \in \MA_K(X)$  which holds in $\MA_K(X \cup Y)$ say on an element $b \in \MA_K(X \cup Y)$. One needs to show that there is an element $a \in \MA_K(X)$ such that $\phi(a,a_1, \ldots,a_n)$ holds in $\MA_K(X \cup Y)$. To do this it suffices to construct an automorphism $\theta$ of $\MA_K(X \cup Y)$ such that $\theta(a_i) = a_i, i = 1, \ldots,n$ and $\theta(b) \in \MA_K(X)$, because in this case $\phi(\theta(a),a_1, \ldots,a_n)$ holds in $\MA_K(X \cup Y)$.  We build  $\theta$ as follows.

Denote by $X_0$ a finite  subset of $X$ such that all monomials of every element $a_i$ are products of elements from $X_0$, i.e., $a_1, \ldots,a_n \in \MA_K(X_0)$. 
There are two cases to consider.

 1) If the set $Y$ is infinite then define  $Y_b$ be a finite subset of $Y$ such that all monomials of the element $b$ are products of elements from $X_0 \cup Y_b$, so $b \in \MA_K(X_0 \cup Y_b)$. Let $X_b$ be an arbitrary subset of $X \smallsetminus X_0$ with the same cardinality as $Y_b$. Now define a bijection $\lambda: X \cup Y \to X \cup Y$ as follows.   $\lambda$ maps:   $X_0$ identically  on $X_0$,   $X\smallsetminus X_0$ bijectively onto $X \smallsetminus (X_0 \cup X_b)$,
$Y_b$ bijectively onto $X_b$, $Y \smallsetminus Y_b$ bijectively onto $Y$. Since $X$ and $Y$ are infinite  such $\lambda$ exists.

2) If the set $Y$ is finite we do the following. Put $Y_b = Y$ and define $X_b$ as above. Take a subset $X_Y $ in $X \smallsetminus  (X_0 \cup X_b)$ of cardinality $|Y|$ and put $X_1 = X \smallsetminus (X_0 \cup X_b \cup X_Y)$, so $X = X_0 \cup  X_b \cup X_Y \cup X_1$. Now a bijection $\lambda: X \cup Y \to X \cup Y$ is defined as follows. 
$\lambda$ maps:   $X_0$ identically  onto  $X_0$, $X_Y$ onto $Y$, $X_b \cup X_1$ onto $X_Y \cup X_1$, $Y_b$ onto $X_b$. Since $X$ is infinite such a bijection $\lambda$ exists. 

The bijection $\lambda$ gives rise to an automorphism $\theta$ of the algebra $\MA_K(X \cup Y)$. Notice that   $\theta (a_i) = a_i, i = 1, \ldots,n$ and $\theta (b) \in \MA_K(X)$, as required.

\end{proof}

 Now we are ready to give first-order classification of free associative algebras over infinite fields.

\begin{theorem} \label{th:Tarski-unitary}
Free associative algebras ${\mathbb A}_{K_1}(X)$ and ${\mathbb A}_{K_2}(Y)$  over  fields $K_1, K_2$, at least one of which is infinite, are elementarily equivalent if and only if the following conditions hold:
\begin{itemize}
\item [1)] either their ranks are  finite and equal or both ranks are infinite; 
\item [2)] the fields $K_1$ and $K_2$ are equivalent in the weak second order logic, i.e.,  $HF(K_1) \equiv HF(K_2)$. 
\end{itemize}
\end{theorem}
\begin{proof}
Suppose, in the notation above, ${\mathbb A}_{K_1}(X) \equiv {\mathbb A}_{K_2}(Y)$.    By Theorem \ref{th:rank-non-com} the sets $X$ and $Y$ are either finite and $|X| = |Y|$ or both infinite.  his proves 1). 

From  Theorem   \ref{th:int-A(X)-in-S(K,N)} and  properties of interpretations (Lemma \ref{co:interp} )  one deduces that $S(K_1, \N) \equiv S(K_2,\N)$, and then from Lemma \ref{le:superstructures} $HF(K_1) \equiv HF(K_2)$, which proves 2).

To show converse, suppose that 1) and 2) above hold. If the ranks of ${\mathbb A}_{K_1}(X)$ and ${\mathbb A}_{K_2}(Y)$
are both infinite then by Theorem \ref{th:elem-submodel}  both algebras have elementary free subalgebras $\MA_{K_1}(X_0)$ and $\MA_{K_2}(Y_0)$ of countable rank, so  $\MA_{K_1}(X_0) \equiv \MA_{K_1}(X)$ and $MA_{K_2}(Y_0) \equiv MA_{K_2}(Y)$. Hence, in this case it suffices to show that $\MA_{K_1}(X_0) \equiv MA_{K_2}(Y_0)$.  This shows that without loss of generality we may assume that the sets either finite or countable and in both cases $|X| =|Y|$. 
By Theorem  \ref{th:int-A(X)-in-S(K,N)} a free  associative algebra $\MA_K(X)$ is 0-interpretable in $S(K,\N)$ uniformly in $K$ and $|X|$, provided that $X$ is either finite or a countable set. Therefore, the condition $S(K_1, \N) \equiv S(K_2,\N)$ implies  that ${\mathbb A}_{K_1}(X) \equiv {\mathbb A}_{K_2}(Y)$. It is left to observe, that  $HF(K_1) \equiv HF(K_2)$ implies that $S(K_1, \N) \equiv S(K_2,\N)$. This proves the theorem.
\end{proof}

Equivalence of the fields $K_1$ and $K_2$ in the weak second order logic is a very strong condition. For example, we mentioned in Section \ref{se:HF-fields} that for some fields $K_1$  the condition $HF(K_1) \equiv HF(K_2)$ implies their isomorphism  $K_1 \simeq K_2$. The theorem above implies that if $K_1$ is a such field and $X$ is a finite set then for any set $Y$ and any field $K_2$ one has $\MA_{K_1}(X) \equiv \MA_{K_2}(Y)$ if and only if $K_1 \simeq K_2$ and $|X| = |Y|$, in which case the algebras ${\mathbb A}_{K_1}(X)$ and ${\mathbb A}_{K_2}(Y)$  are isomorphic.

\begin{cor} 
If $X$ is a finite set and a field $K_1$ is  one of the fields from Section 2.3, then the algebras ${\mathbb A}_{K_1}(X)$ and ${\mathbb A}_{K_2}(Y)$ are elementarily equivalent if and only if  they are isomorphic.
\end{cor}

\section{Rings elementarily equivalent to $\MA_K(X)$}

In this section we study arbitrary rings $B$ which are elementarily equivalent to $\MA_K(X)$. As it was  mentioned in the introduction it is usually very difficult to describe all such $B$ unless some reasonable restrictions on $B$ are imposed. Here we assume that $B$ satisfies a rather  weak finitary condition, namely that $B$ has a proper centralizer which is Noetherian. 
Note, that  any ring $B$ elementarily equivalent to $\MA_K(X)$ must be a central algebra over a filed where each proper centralizer is commutative (see the argument below), therefore the class  of rings $B$ under consideration contains, for example, all central algebras which have a proper centralizer that is commutative and finitely generated as an algebra. 

There are many examples of rings $B$ as above (for example, free associative algebras $\MA_K(X)$), however, we do not know whether or not an arbitrary  finitely generated central $K$-algebra with all proper centralizers commutative has a proper Noeterian centralizer.

In this section we prove the following principal result.

\begin{theorem} \label{th:elem-classif-unital}
Let $K$ be an infinite field and $X = \{x_1, \ldots, x_n\}$ a finite set. Assume that $B$ is an arbitrary ring that has  a proper Noetherian centralizer. Then $\MA_K(X) \equiv B$ if and only if $B$ satisfies the following conditions:
\begin{itemize}
\item the center of $B$ is a field, say $K_1$,  in particular  $B$ is a central $K_1$-algebra;
\item as a  $K_1$-algebra $B$ is isomorphic to a free associative algebra  $\MA_{K_1}(Y)$;
\item   $HF(K) \equiv HF(K_1)$ and $|X| = |Y|$.
\end{itemize}
\end{theorem}
\begin{proof}  
Let $B$ be a ring that has  a proper Noetherian centralizer  and $\MA_K(X) \equiv B$. 

Notice, that the field $K$ is the center  and  the maximal ring of scalars of  $\MA_K(X) $ (for the latter see Proposition \ref{th:scalar}).   The center of $\MA_K(X) $ is obviously interpretable in $\MA_K(X) $. Therefore, the center in $B$ is also a field, which we denote by $K_1$. By Theorem \ref{th:bilinear}  the maximal ring of scalars of  $\MA_K(X) $  is interpretable in $\MA_K(X) $ uniformly in the size of the finite complete system and the width of the multiplication (viewed as a $K$-bilinear map). It is clear that existence of a complete system  of a given size can be written by a sentence of the ring language, as well as the width of the multiplication. Hence the same formulas that interprets the maximal ring of scalars in $\MA_K(X) $ will interpret the maximal ring of scalars in $B$. Using this interpretation one can write down that the maximal ring is a field. Hence it is isomorphic to the center of $B$. In fact, one  can also write down a sentence that states that the center is the maximal ring of scalars in $\MA_K(X) $ (it suffices to write down that every element in the center is obtained from the identity 1 by  multiplication by  an element from the maximal filed) hence in $B$. This proves 1).

To prove 2) we show most of the objects proved   in Section \ref{se:Interpet_free assoc} to be interpretable in $\MA_K(X)$ are also interpretable and by the same formulas in $B$.  Indeed, notice first that by Theorems  \ref{th:centralziers} and  \ref{th:def-centr-iso} all proper centralizers of $B$ are definably isomorphic to each other as rings. Furthermore, this common ring, say $C$ is 0-interpretable in $B$ by the same formulas that the ring of one-variable polynomials $K[t]$ is interpretable in $\MA_K(X)$, it follows from the properties of 0-interpretations 
(see Lemma \ref{le:interpr_corol}) that $C \equiv K[t]$. Since in the ring $B$  at least one proper centralizer of $B$ is Noetherian the the ring $C$ is Noetherian. By Theorem \ref{th:Bauval_poly} the ring $C$ is isomorphic to $K_2[t]$ for some field $K_2$. Since $K_1$ is the set of all invertible (and $0$) elements in $C$ it follows that $K_2 = K_1$.  Thus, we showed that every proper centralizer of $B$ is isomorphic to $K_1[t]$. 

Now all the statements 1), 2), 3) of Lemma \ref{le:non-comm-arbitrary-char} hold in $B$ and the formulas that give the corresponding interpretations of  arithmetic $\N$ are exactly the same as in Lemma \ref{le:non-comm-arbitrary-char}.  Indeed, 1) holds because every proper centralizer of $B$ is isomorphic to $K_1[t]$, and the formulas used in the interpretations of $\N$ as $\N_P$ are uniform in the field $K$ (or $K_1$). To prove  2) it suffices to notice that since the isomorphisms of the interpretations $\N_P$ and $\N_Q$ of arithmetic  in $\MA_K(X)$ are uniformly definable by some formulas, say  $\Lambda$,  one can write down the condition that these formulas $\Lambda$ indeed give an isomorphism between the interpretations. Therefore the corresponding interpretations in $B$ will be also definably isomorphic, so 2) holds in $B$. 3) follows from 2) as usual.  This gives uniform interpretation of arithmetic $\N$ in $B$ precisely by the same formulas as in Lemma \ref{le:non-comm-arbitrary-char}.

 A similar argument shows that all statements of Theorem  \ref{th:S(F,N)-non-comm} also hold in $B$, and the corresponding interpretations are given precisely by the same formulas as in $\MA_K(X)$. This gives  interpretations of $S(\N,K_1)$ in $B$ which satisfy all the statements of Theorem  \ref{th:S(F,N)-non-comm}.
 
 The formula $Gen(V)$ of ring language, where $V = \{v_1, \ldots,v_n\}, n = |X|$,  from  Theorem  \ref{th:bases}  defines in $\MA$ the set of all free bases. It follows that $\exists V Gen(V)$ holds in $B$, say on a tuple $Y = \{y_1, \ldots,y_n\}$. Fix this tuple $Y$ in $B$ as a tuple of parameters (it  plays the same part in formulas of interpretations in $B$  as $X$ plays in $\MA$).

 Now we show that a direct  analog of Lemma \ref{le:M-X-c} holds in $B$. Indeed,  by Lemma \ref{le:monomials} the submonoid $K\MM_X$ of $\MA$  is defined in $\MA$ by a formula $\phi(a,X)$ with parameters $X$. Let $M_Y$ be a subset of $B$ which is defined in $B$ by the formula $\phi(a,Y)$ with variable $a$ and parameters $Y$.  $M_Y$ is a multiplicative submonoid of $B$  since $K\MM_X$ is a multiplicative submonoid of $\MA$ and $K\MM_X \equiv M_Y$.  $K\MM_X$ contains the field $K$ and this can be written by formulas with parameters $X$  because $K$ and $K\MM_X$ are both definable in $\MA$. Hence the submonoid $M_Y$ contains the  field $K_1$.  By Lemma \ref{le:monomials}  the monoid $K\MM_X/\sim$ is interpretable in $\MA$ with parameters $X$ and is isomorphic to the free monoid $\MM_X$.
 Lemma \ref{le:M-X-c} tells us that there is a definable isomorphism between the submonoid $\MM_X$ in $\MA$ and the free submonoid $\MM_{X,c}$ (here $c$ is an arbitrary element of $X$, say $c = x_2$) canonically interpreted in $S(\N,\N)$ (see the paragraph before Lemma \ref{le:M-X-c}). Since the analog of Theorem \ref{th:S(F,N)-non-comm} holds in $B$ the same formulas as in Lemma \ref{le:M-X-c}  interpret the structure $S(\N,\N)$ in $B$, hence the formulas that interpret the free monoid  $\MM_{X,c}$ in $S(\N,\N)_P$ in $\MA$   interpret a free monoid $\MM_{Y,v_2}$ isomorphic to $\MM_{X,c}$ in the corresponding interpretation of $S(\N,\N)$ in $B$. Therefore, the monoid $M_Y/\sim$ which is interpreted in $B$ is definably isomorphic to  the free monoid $\MM_{Y,v_2}$. Observe that, as in the case of $\MM_X/\sim$, the images of the elements from $Y$ in $M_Y/\sim$ form a basis of $M_Y/\sim$. It follows that the elements from $Y$ also generate a free monoid, which is isomorphic to $M_Y/\sim$ under the canonical projection $M_Y \to M_Y/\sim$. We denote this monoid $\MM_Y$. Obviously in this case $M_Y = K_1\MM_Y$. Direct inspection of the argument in Lemma \ref{le:M-X-c} shows that the formula $\Phi(y,z,X,c)$  described in this lemma is such that the formula  $\Phi(y,z,Y,v_2)$ obtained from $\Phi$ by replacing $X$ with $Y$ and $c$ with $v_2$ holds in $B$ on a pair of elements  $t,u$  if and only if $t \in T_{v_2}$ and $u = \alpha M_t$ for some $\alpha \in K_1$ (here we use notation from Lemma  \ref{le:M-X-c} adopted to the corresponding objects in $B$).  
 This shows that Lemma \ref{le:M-X-c}, with the appropriate adjustments in notation,  holds in $B$.
 
 Similarly, Lemma \ref{le:s-and-f} holds in $B$ after proper adjustment of notation. But then the argument from Theorem \ref{th:bases}  is valid in $B$ as well, and this shows that $B$ is a free associative algebra over the field $K_1$ and the formula $Gen(V)$ defines in $B$ the set of bases. In particular $Y$ is basis of $B$. Thus $B = \MA_{K_1}(Y)$ and $\MA_K(X) \equiv \MA_{K_1}(Y)$. By Theorem  \ref{th:Tarski-unitary} $HF(K) \equiv HF(K_1)$, so 3) follows. 
This proves the theorem.

\end{proof}

\section{Non-unitary free associative algebras} \label{se:non-unitary}

Let $K$ be a field and $X$ a non-empty finite set. Denote by ${\mathbb A}^0_K(X)$ 
a free associative algebra with basis $X$ without unity. One can view  elements in ${\mathbb A}^0_K(X)$ as linear combinations over $K$ of non-commutative monomials on $X$.  In this section we prove that the algebra ${\mathbb A}^0_K(X)$ has very  similar model theoretic properties as the free associative algebra $\MA_K(X)$. 

Recall (Theorem \ref{th:field}) that  the field $K$ and its action on  ${\mathbb A}^0_K(X)$ are 0- interpretable 
in ${\mathbb A}^0_K(X)$ uniformly in $K$. This fact allows one to prove the following result that is crucial in our study of model-theoretic properties of ${\mathbb A}^0_K(X)$.

\begin{theorem} \label{co:non-unitary}
Algebra ${\mathbb A}_K(X)$ is 0-interpretable  in ${\mathbb A}^0_K(X)$ uniformly in $K$ and $|X|$.
\end{theorem}
\begin{proof}
 Notice that ${\mathbb A}_K(X) \simeq 1 \cdot K \oplus {\mathbb A}^0_K(X)$. 
 
 Assume $|X| \geq 2$. As we noted above, due to Theorem \ref{th:field} the field $K$ and its action on $ {\mathbb A}^0_K(X)$ are 0-definable in $ {\mathbb A}^0_K(X)$. This allows one to interpret ${\mathbb A}_K(X) $  in ${\mathbb A}^0_K(X)$ as the set $K \times {\mathbb A}^0_K(X)$ and interpret the scalar multiplication by $K$ on this sets by formulas of ring theory. This gives an 0-interpretation of ${\mathbb A}_K(X)$  in ${\mathbb A}^0_K(X)$.

If $|X| =1$ then the field $K$ and its action is 0-interpretable in  ${\mathbb A}^0_K(X)$ by Remark \ref{re:def-field}.
\end{proof}

\subsection{Decidability and first-order classification}

Theorem \ref{co:non-unitary} allows one to reduce the Tarski's problems and the  elementary classification problem  for algebras ${\mathbb A}^0_K(X)$ to the corresponding problems for algebras ${\mathbb A}_K(X)$.

As a corollary of Theorems \ref{th:undecidable} and \ref{th:Tarski-unitary} for the unitary case and Theorem  \ref{co:non-unitary} one gets  the following results.

\begin{theorem}\label{th:undecidable-non-unitary}
The first-order theory of ${\mathbb A}^0_K(X)$ is undecidable for any filed $K$ and any  non-empty finite set $X$.
\end{theorem}
 
\begin{theorem} \label{th:Tarski-non-unitary}
Free associative non-unitary algebras ${\mathbb A}^0_{K_1}(X)$ and ${\mathbb A}^0_{K_2}(Y)$  of finite rank over infinite fields $K_1, K_2$ are elementarily equivalent if and only if their ranks are the same and $HF(K_1)\equiv HF(K_2).$
\end{theorem}

\begin{theorem} \label{th:elem-classif-non-unital}
Let $K$ be an infinite field and $X = \{x_1, \ldots, x_n\}$ a finite set. Assume that $B$ is a ring with   a proper Noetherian centralizer. Then $\MA_K^0(X) \equiv B$ if and only if $B$ satisfies the following conditions:
\begin{itemize}
\item the maximal ring of scalars of $B$ is a field, say $K_1$, in particular,  $B$ is a  $K_1$-algebra;
\item as an  $K_1$-algebra $B$ is isomorphic to a free associative algebra  $\MA_{K_1}^0(Y)$;
\item   $HF(K) \equiv HF(K_1)$ and $|X| = |Y|$.
\end{itemize}
\end{theorem}
\begin{proof}
Let $B$ be a ring with   a proper Noetherian centralizer  such that $\MA_K^0(X) \equiv B$.

Notice that the same formulas that in Theorem \ref{co:non-unitary} interpret $\MA_K(X)$ in $\MA_K^0(X)$ 
also interepret a ring $B_1 = 1\cdot K_1 \oplus B$ in $B$. Since $\MA_K(X)^0 \equiv B$ it follows that $\MA_K(X) \equiv B_1$. Now to apply Theorem \ref{th:elem-classif-unital} in this case one needs to show that $B_1$ is also a ring with   a proper Noetherian centralizer. But this is obvious because if $b \in B$ is an element such that the centralizer $C_B(b)$ is proper in $B$ and Noetherian then the centralizer $C_{B_1}(b) = 1\cdot K_1 \oplus C_B(b)$ is also Noetherian.  By Theorem \ref{th:elem-classif-unital} $B_1 \simeq \MA_{K_1}(X)$. It remains to be seen that in this case $B \simeq \MA_{K_1}^0(X)$. Since the field and its action is definable in $\MA_K^0(X)$,  the set of free bases  of $\MA_K^0(X)$ can be defined  in the theory of $\MA_K^0(X)$ as in Theorem \ref{th:bases}.   Notice that in this case we know from the interpretation we use that there exist bases which all their components belong  to the direct summond $\MA_K(X)^0$ of  the interpretation $1 \cdot K \oplus {\mathbb A}^0_K(X)$ of $\MA_K(X)$  in $\MA_K(X)^0$ from Theorem \ref{co:non-unitary}. This can be written by formulas in $\MA_K(X)^0$. It is easy to see by direct  inspection that the same formulas will define some bases of $B_1 \simeq \MA_{K_1}(X)$ that belong to $B$. Hence $B$ is isomorphic to $\MA_K(X)^0$, as claimed.
 \end{proof}

\subsection{Arbitrary rings elementarily equivalent to $\MA_K^0(X)$}\label{se:finite-width}

Let $A$ be an associative ring.  For $n \in N$ denote by $A^n$ the $n$-th power of $A$, i.e., the subgroup of the additive group $A^+$ generated by all the products of the type $a_1 \ldots a_n$, where $a_i \in A$. In fact,  $A^n$ is a (two-sided) ideal of $A$. If $A$ is an algebra over a field $K$, then $A^n$ is the subspace generated by the products $a_1 \ldots a_n$. 
 $A$ is {\em nilpotent} of nilpotency class $c$ if $A^c\neq 0$, but $A^{c+1} = 0$, and $A$ is {\em residually nilpotent} if $\bigcap_{n=1}^\infty A^n = 0$. 

\begin{definition} We say that $A^n$ has finite width if there is a positive integer $k$ such that every element $a$ in $A^n$ is a sum of at most $k$ products of the type $a_1 \ldots a_n$, where $a_i \in A$.
The least such $k$ is termed the {\em width} of $A^n$ (denoted by $width(A^n)$).
\end{definition}

\begin{lemma}  \label{le:19}
The following holds:
\begin{enumerate}
\item [1)] for any $n,k \in \N$ there is a formula $\phi_{n,k}(y)$ that  defines without parameters the ideal $A^n$ in any ring $A$  with $width(A^n) = k$; 
\item [2)] there exists a first-order sentence  $W_{n,k}$ of ring theory such  that for any ring $A$
 $$
 A \models W_{n,k} \Longleftrightarrow width(A^n) = k.
 $$
\end{enumerate}
\end{lemma}
\begin{proof}  
Put
$$
\phi_{n,k}(y) = \exists w_{11},\ldots w_{1n} \ldots w_{k1} \ldots w_{kn}(y=\Sigma _{j=1}^k 
 w_{j1} \ldots w_{jn}).
$$
 It follows from the definitions that if $width(A^n) = k$ in some ring $A$ then $\phi_{n,k}(y)$ defines $A^n$ in $A$. This proves 1).

To show 2) consider a sentence  
$$
\psi _{n,k}=\forall y ( \phi_{n,k+1}(y) \rightarrow \phi_{n,k}(y))
$$
 which states that any sum of $k+1$  $n$-products of  elements in a ring is in fact a sum of $k$ $n$-products of elements. Clearly, for any ring $A$ 
$$
width(A^n)\leq k \Longleftrightarrow A \models \psi _{n,k}.
 $$
 Therefore 
 $$
 width(A^n) = k \Longleftrightarrow A \models \psi _{n,k} \wedge \neg \psi _{n,k-1} .
 $$
\end{proof}

Now let  $\MA = \MA_K^0(X)$.   Set $r= |X|$.

\begin{lemma} \label{width} The following holds in $\MA_K^0(X)$:
\begin{itemize}
\item [1)] For any $n$, $width(\MA^n)\leq r^n$;
\item [2)] $\MA$ is residually nilpotent, i.e., $\bigcup_{n = 1}^\infty \MA^n = 0$;
\item [3)] for any $n$ $\MA/\MA^n$ is a free  nilpotent associative algebra $N_{n,K}(X)$ over $K$ of class $n$  and rank $r$.
\end{itemize}
\end{lemma}
 \begin{proof}   Every monomial of degree at least $n$ begins with the product of $n$ letters from the basis $X$. So collecting all summands $w_{j1} \ldots w_{jn}$ in an element $a = \Sigma _{j=1}^k w_{j1} \ldots w_{jn}$ from $\MA^n$ with the same initial product one gets a sum of at most $r^n$ products.
 
 2) and 3)  are well known, can be found, for example, in \cite{poly}.

 \end{proof}

Recall (see, for example \cite{Baum67a,Baum67b}) that a $K$-algebra $R$ is {\em para-free} if it is residually nilpotent and for any $n \in\mathbb{N}$ $R/R^n \simeq \MA/\MA^n$ as $K$-algebras.

\begin{theorem} \label{th:width}
If a ring $B$ is elementarily equivalent to a free associative algebra ${\mathbb A}^0_K(X)$ of rank $n$ ,  then $B$ is an associative algebra over a field $K_1$, such that:
\begin{itemize}
 \item $K_1$ is elementarily equivalent to $K$,
\item $B/B^n\equiv C_n$, where $C_n$ is a free $n$-nilpotent associative algebra with basis $X$ over the field $K_1$.\end{itemize}
In particular, if $B$ is residually nilpotent, then $B$ is para-free.
\end{theorem}
\begin{proof}  Lemmas  \ref{width} and \ref{le:19} imply that the ideals $\MA_K(X)^n$  are 0-definable in ${\mathbb A}^0_K(X)$ by formulas of the ring language and that  the same formulas define in the ring $B$ the ideals  $B^n$. Hence from the properties of 0-interpretations it follows that  $B/B^n\equiv \MA/\MA^n$ for every $n$.  By Lemma \ref{width}  the algebra $\MA/\MA^n$ is a free  nilpotent associative algebra $N_{n,K}(X)$ over $K$ of class $n$  and rank $r = |X|$.   It was shown in \cite{M1} that in this case $B/B^n$ is  a free  nilpotent associative algebra $N_{n,K_1}(X)$ over $K_1$ of class $n$  and rank $r = |X|$, where $K \equiv K_1$. This proves the theorem.

\end{proof}

The following result shows that there non-para-free algebras that are  elementarily equivalent to $\MA_K^0(X)$.
\begin{theorem}\label{th:non-res-nilpotent}
There is a countable  not residually nilpotent algebra $B$ such that $B \equiv \MA_K^0(X)$.
\end{theorem}
\begin{proof} Consider $\MA_K^0(X)$ with $|X| = r$.
Consider the following infinite set of formulas of the first-order language of ring theory in one variable $y$:
$$
\Phi  = \{  \phi_{n,r^n}(y) \mid n \in \N\} \cup \{y \neq 0\} 
$$
It is clear that any finite subset  of formulas from $\Phi$ can be satisfied in $\MA_K^0(X)$  on some particular element. Indeed, every finite subset $\Phi_0 \subset \Phi$  states that $y$ belongs to $\MA^n$, where 
$n$ is the largest index that occur in the formulas $\phi_{n,r^n} \in \Phi_0$. It follows that $\Phi$ is a set of formulas that is locally consistent with the theory $Th(\MA_K^0(X))$, i.e., it is 1-type in $Th(\MA_K^0(X))$. Therefore there is a countable model $B$ of $Th(\MA_K^0(X))$ that realizes this type, say on an element $b \in B$. Then $B \equiv \MA_K^0(X)$  and $0\neq b \in \bigcap_{n = 1}^\infty B^n$.

\end{proof}

\section{Some open problems for free associative algebras}

\begin{problem}
 Are free associative algebras $\MA_K(X)$ equationally Noetherian?
 \end{problem}
 Recall that a ring $R$ is called equationally Noetherian if every infinite system of equation in finitely many variables with constants from $R$  is equivalent over $R$ (has the same solution set) to some of its finite subsystems.

\begin{problem}
Is it true that any finitely generated central algebra where all proper centralizers are commutative has at least one Noetherian  proper centralizer.
\end{problem}

If the answer to the problem above is affirmative then Theorem \ref{th:elem-classif-unital} will give a description of all such algebras that are elementarily equivalent to $\MA_K(X)$.

 In view of Theorems \ref{th:width} and \ref{th:non-res-nilpotent}  in Section \ref{se:finite-width} the following problem is of prime interest in our study of rings which are elementarily equivalent to $\MA_K^0(X)$.
 \begin{problem}
 Describe  para-free central algebras that are elementarily equivalent to $ \MA_K^0(X)$
 \end{problem}

\end{document}